\documentclass[a4paper,9pt,reqno]{amsart}
\usepackage{hyperref}
\hypersetup{
    colorlinks,
    citecolor=blue,
    filecolor=blue,
    linkcolor=blue,
    urlcolor=blue
}
\usepackage{amsmath}%
\usepackage{amsfonts}%
\usepackage{amssymb}%
\usepackage{graphicx}
\usepackage{enumitem}
\usepackage{comment}
\newtheorem{theorem}{Theorem}[section]
\theoremstyle{plain}

\newtheorem{corollary}[theorem]{Corollary}

\newtheorem{definition}[theorem]{Definition}
\newtheorem{example}[theorem]{Example}

\newtheorem{lemma}[theorem]{Lemma}

\newtheorem{proposition}[theorem]{Proposition}
\newtheorem{remark}{Remark}

\numberwithin{equation}{section}

\newcommand{\R}{\mathbb{R}}
\newcommand{\N}{\mathbb{N}}

\newcommand{\E}{\mathbb{E}}
\newcommand{\Var}{\mathrm{Var}}

\newcommand{\T}{\mathcal{T}}

\newcommand{\s}{X}

\newcommand{\mesh}{\delta}
\newcommand{\m}{\delta}

\begin{document}
\title[Stein's method  with bespoke derivatives]{One-dimensional Stein's method  with bespoke derivatives}
\author{Gilles Germain}
\address[]
{Gilles Germain, Université libre de Bruxelles, Département de Mathématique, Campus Plaine, Boulevard du Triomphe
CP210, B-1050 Brussels}
\email[]{gilles.germain@ulb.be}%
\author{Yvik Swan}
\address[]{Yvik Swan, Université libre de Bruxelles, Département de Mathématique, Campus Plaine, Boulevard du Triomphe
CP210, B-1050 Brussels}
\email[]{ yvik.swan@ulb.be}

\begin{abstract}
We introduce a  version of  Stein's method of comparison of operators specifically tailored to the problem of   bounding the Wasserstein-1 distance between continuous and discrete distributions on the real line. Our approach rests on a new family  of weighted discrete derivative operators, which we call bespoke derivatives.  
We also propose new bounds on the derivatives of the solutions of  Stein equations for Integrated Pearson random variables; this is a crucial step in Stein's method. We apply our result to several examples, including the Central Limit Theorem, Polya-Eggenberger urn models, the empirical distribution of the ground
state of a many-interacting-worlds harmonic oscillator, the stationary distribution for the number of genes in the Moran model, 
and the stationary distribution of the Erlang-C system. Whenever our bounds can be compared with bounds from the literature, our constants are sharper. 
\end{abstract}
\maketitle

\section{Introduction and main results} 
\label{sec:1}

 In this paper, we use   Stein's method   to propose new   bounds on the Wasserstein distance 
$W_1(X,Z):=\int_{-\infty}^{+\infty} | \mathbb{P}[X \le z] - \mathbb{P}[Z \le z] | dz$ (also known as the  Kantorovitch or $L^1$ distance)  
between the laws of real random variables $X$ and $Z$ under the assumption that  $X$ is  discrete   and $Z$ is continuous. The use of Stein's method for this purpose   has attracted much attention. For instance \cite{chen2023optimal} use the \emph{zero-bias transform} to study  normal approximation for the empirical distribution of the ground state of a many-interacting-worlds harmonic oscillator.   \cite{chatterjee2011nonnormal}   use  \emph{exchangeable pairs} to study   exponential approximation of   the uniform measure on the set of eigenvalues of the Bernoulli-Laplace Markov chain. In the same vein,  \cite{dobler2015stein} also uses an {exchangeable pairs} approach to study  the distance between the beta   and  the Polya-Eggenberger distribution, while  \cite{goldstein2013stein} tackle the exact same problem using a  \emph{comparison of operators approach}. This last reference is the  main source of inspiration behind the present work, and here we propose to revisit their approach using a new family of discrete derivative operators which we call ``bespoke derivatives''.

The   main theoretical  contribution of the present paper is Theorem \ref{theo main intro} below, which contains a  general abstract bound permitting to tackle all the aforementioned problems, and many more.   Before stating the theorem, we state the  standing  Assumption under which we shall work for the entire paper. 

\medskip

\noindent \textbf{Assumption.} The   real random variable $Z\sim q$ has  probability density function $q : \R \to \R^+$ such that  $\{x \in \mathbb R \, | \, q(x)>0 \}$ is an interval with interior $(a, b)$ for $- \infty \le a < b \le +\infty$ and   $q$  is absolutely continuous on every compact interval included in $(a,b)$; also  $\E[|Z|]<\infty$, i.e.\ the identity function $\mathrm{Id}$ belongs to  $L^1(q)$.

The   real random variable $X\sim p$ is  a discrete random  variable with mass function $(p_i)_{i \in I} \subset (0, 1)$ on a finite (resp.,  countable)  set of points $(x_i)_{i  \in I} \subset \overline{(a,b)}$ such that  $x_i < x_{i+1}$   for  $I = \{0, 1, \ldots, \ell \}$ ($\ell \in \mathbb N$) (resp.,  $I = \mathbb N$). 

\medskip

 Under this  Assumption (which,  along with the notations it contains, shall not be repeated) we obtain the following result.  

 \begin{theorem}[Uniform Wasserstein bounds]
\label{theo main intro} 

Let  $w : (a, b) \to \mathbb R^+_0$ be two times differentiable, set  $s(x) = (q(x) w(x))'/q(x)$ for all $x \in (a, b)$ and define  the weight sequence 
 $(\pi_i)_{i \in I} =: (\pi(x_i))_{i \in I} $  through
\begin{align}
\label{eq_ai3-intro}
 \pi_i&=\frac{1}{p_i w_i}\sum_{j=0}^i \left((x_j-x_i)s_j+w_j\right)p_j, 
\end{align}
with $w_j = w(x_j)$ and $s_j = s(x_j)$ for $j \in I$.

Let  $\tau_q$ be the Stein kernel of $q$ (see equation \eqref{steinkk}). We make the following assumptions on $q$ and $w$:  $\tau_q/w$ is bounded on $(a, b)$,     $s\in L^1(q)$, $s$ is strictly decreasing and continuous,  $\E[s(Z)]=0$. We also assume that  $w\in L^1(p)$,   $\pi w\in L^1(p)$,  and 
\begin{align}
\label{condition1-intro}
\E[s(X)]=0\quad\text{and}\quad \E\left[X s(X) + w(X) \right]=0.
\end{align}   
Then  
\begin{align}
W_1(X,Z)&  \leq C^0_q(w)  \, \mathbb{E}\left[|\pi(X)| \mesh^+(X) + |1-\pi(X)| \mesh^-(X) \right]\nonumber\\
&\quad +C_q^1(w)  \, \mathbb{E}\left[|\pi(X)| (\mesh^+(X))^{2} + |1-\pi(X)| (\mesh^-(X))^2 \right]\label{elegantform-intro}
\end{align}
where  
\begin{itemize}
    \item   $\mesh^-(x_i) = x_{i}-x_{i-1}$ for all $i \in I\setminus \{0\}$, and $\mesh^+(x_i) = x_{i+1}-x_i$  for all $i \in I$ (or $I \setminus \{\ell\}$ when $\ell \in \mathbb N$);
\item $C_q^0(w)$ and $C_q^1(w)$ are  absolute constants depending only on $q$ and $w$ (see Section~\ref{sn:abstractwass}).
\end{itemize} \end{theorem} 

Note that the sequence of weights defined in \eqref{eq_ai3-intro} satisfies $\pi_0 = 1$ and, if $\ell < \infty$, $\pi_\ell = 0$. We have the following immediate consequence of Theorem \ref{theo main intro}. 

\begin{corollary}\label{cor:regsp}
If furthermore $x_i - x_{i-1} = \mesh>0$ for all $i \in I\setminus \{0\}$    then 
\begin{align} \label{simplllfir}
    W_1(X, Z) \le \left(C_q^0(w) \mesh +C^1_q(w)\mesh^2 \right) \E[|\pi(X)|   + |1-\pi(X)|].
\end{align}
If moreover  $0 \le \pi_i    \le 1$  for all $i \in I$ then 
\begin{align} \label{simplll}
    W_1(X, Z) \le C_q^0(w) \mesh +C^1_q(w)\mesh^2.
\end{align}
\end{corollary} 

 As we shall soon see, the condition $0 \le \pi_i \le 1$ for all $i \in I$  is  satisfied in many (but not all) interesting situations. Since the mesh $\mesh$ ought to be small    in all our intended applications (indeed we approximate discrete distributions by continuous ones) this simple corollary indicates  that our results seem to exhibit the correct behavior.

Condition \eqref{condition1-intro} is  a standardisation  relating certain moments of $X$ to those of $Z$. In practice this condition comes at no loss of generality though 
in specific applications it may result in Theorem \ref{theo main intro} giving  slightly different bounds than those from the literature. 
For comparison purposes, it  will   be useful to make note of the inequality  
\begin{equation}\label{rema1}
W_1(b_1X+c_1, b_2X+c_2) 
\leq |b_1-b_2|\E|X|+|c_1-c_2|
\end{equation}
which will allow us to switch between standardizations. 
 
The constants $C_q^0(w)$ and $C_q^1(w)$ in \eqref{elegantform-intro}  are  a form of  ``Stein  factors'' which are  familiar features of bounds obtained through  Stein's method; see Definition \ref{defn:2.10}.  We will study  these constants in Section~\ref{sn:abstractwass}. For instance, we will show that when $q = \varphi$ is the standard normal distribution then $C_\varphi^0(1) = 1$ and $C_\varphi^1(1) = 0$, whereas when $q$ is the density of an exponential distribution with parameter $\lambda>0$ and $w(x) = x/\lambda$ then  $C_q^0(w) = 3/2$ for all $\lambda>0$ and $C_q^1(w) = 0$. 
%
%

Finally, the weights $(\pi_i)_{i \in I}$ defined in \eqref{eq_ai3-intro} are the main novelty of the paper; we will study them in detail  in Sections \ref{sec:2} and \ref{sn:applications}. We will give many examples where these sequences are perfectly explicit, so that bounds such as \eqref{elegantform-intro} and its offshoots can be easily computed in all these contexts. For cases where  the weights are not available explicitly, sufficient conditions guaranteeing $0 \le \pi\le 1$ will be provided in Section \ref{sn:more}.

Obviously, the quality of bound \eqref{elegantform-intro} hinges on the choice of the function $w$ which is crucial both for the value of the constant $C_q^j(w)$ $j=0,1$ and for the behavior of the weighting sequence $(\pi_i)_{i \in I}$.   
%
When $Z\sim q$ has finite mean, a natural choice of test function $w$ in Theorem \ref{theo main intro} is the \emph{Stein (covariance) kernel} defined for $x \in (a, b)$ as 
\begin{equation}\label{steinkk}
    \tau_q(x) = \frac{1}{q(x)} \int_x^b (u - \mathbb{E}[Z]) q(u) du
\end{equation}
and 0 elsewhere. 
 The function $\tau_q$ is well-studied within the context of Stein's method, see e.g.\ \cite{ernst2020first, saumard2019weighted} for  recent overviews. 
  In particular, since  in this case it holds that $s(x) = {(\tau_q(x)q(x))'}/{q(x)} = -(x - \mathbb{E}[Z])$,  the assumptions on $q$ in Theorem \ref{theo main intro} are trivially satisfied, and conditions  \eqref{condition1-intro} then read as $\mathbb E[X] = \mathbb E[Z]$ and $\mbox{Var}[X] = \mathbb{E}[\tau_q(X)]$. 
  We note that when $Z$ is Pearson distributed with finite variance, its Stein kernel is polynomial of degree at most 2 (see e.g.\ \cite{stein1986approximate}). Such laws are called integrated Pearson (see \cite{Afendras}) and direct computations show that then, for $w = \tau_q$, condition \eqref{condition1-intro} simplifies to 
\begin{equation*}
\E[X]=\E[Z]\quad\text{and}\quad\Var[X]=\Var[Z].
\end{equation*}
We shall also obtain very simple bounds on the constants $C_q^j(\tau_q), j = 0, 1$ in this case.

\medskip 

To showcase the kind of results  we can obtain from  Theorem \ref{theo main intro}, we propose the following examples (many more examples will be provided in Section \ref{sn:applications}).

\begin{example}[P\'olya-Eggenberger urn models]\label{ex:1}
Let $Y$ be distributed according to the P\'olya-Eggenberger  distribution with parameters $\alpha$, $\beta$, $m$ and $n$. Then with $X = Y/n$ and $Z \sim \mathrm{Beta}(\alpha/m, \beta/m)$   it holds that 
\begin{equation*}
    W_1(X, Z) \le \frac{C(\alpha, \beta, m)}{n}
\end{equation*}
with $C(\alpha, \beta, m)$ a function of the parameters  given in Proposition \ref{prop:ourgoldrein} below. This bound is of the exact same form and order as the one in    \cite{goldstein2013stein}; only our constants differ. See  Proposition~\ref{prop:ourgoldrein} and Remark \ref{rem:polyaeg} for more details.  
\end{example}

\begin{example}[Bernoulli-Laplace Markov chain]\label{ex:2}
    Let $Y$ be a random eigenvalue chosen randomly from the set 
eigenvalues of the
Bernoulli-Laplace Markov chain. Then with  $X = \frac{n}{2} Y + 1$ and $Z \sim \mathrm{Exp}(1)$ it holds that 
\begin{equation*}
    W_1(X, Z)  \le\frac{3\sqrt2}{\sqrt{n}} + \frac{3}{n} \le \frac{5}{\sqrt n}.
\end{equation*}
This is to be compared with \cite{chatterjee2011nonnormal} where a Wasserstein bound in exactly the same problem  was obtained (through an entirely different take on Stein's method) with the same rate but a  constant equal to 12. See Proposition \ref{prop:chatt} for more details.
\end{example}

    \begin{example}[Normal approximation of a hypergeometric distribution]\label{ex:3}
    Let $Y$ be the number of marked balls in a sample of  $r$ balls taken from  an urn with $N$ balls, $n\le r$ of which are marked; $Y$ is a hypergeometric random variable with mean $\mu = rn/N$ and variance $\sigma^2= nr(N-r)(N-n)/((N-1)N^2)$ (we suppose that $N > n+r$).  Let $Z \sim \mathcal{N}(0, 1)$. Then with $X = (Y - \mu)/\sigma$ it holds that $$W_1(X,Z) \le \frac{1}{\sigma},$$
which concurs with the known fact (already mentioned in  Feller \cite{feller1945normal}) that the standardized hypergeometric  and standard Gaussian become indistinguishable if and only if  $\sigma \to \infty$, see e.g.\ \cite{nicholson1956normal, lahiri2007berry}. More details are provided in Example \ref{exa:hyperggg}. Quantitative bounds for normal approximation of a hypergeometric distribution are available  in  Kolmogorov distance (i.e. supremum norm of the difference of cumulative distribution functions) in the symmetric case \cite{mattner2018normal}.  
 
\end{example}

\medskip 

  The rest  paper is organised as follows. After introducing some notations in  Section~\ref{sec:notations}, in Section   \ref{sec:2.3} we set up our version of Stein's method of comparison of operators with bespoke derivatives (introducing the different  operators, the Stein equations, etc.) leading to a general abstract bound stated in Theorem \ref{th:steincomp}.  In Section \ref{sec:besp} we explain the origin of the weights $(\pi_i)_{i \in I}$
  and  in Section \ref{ex:xomepisnorm} we provide explicit expressions of these weights  when the target distribution $q$  is the standard Gaussian and $w= 1$ (which is the Stein kernel in this case). In Section  \ref{sn:abstractwass} we define and study the Stein  factors $C_q^j(w), j = 0, 1$, and also provide explicit bounds when the target $q$ is Integrated Pearson and $w = \tau_q$ is the corresponding Stein kernel. In Section \ref{sn:proof} we prove the  Wasserstein bound from  our main Theorem~\ref{theo main intro}. In Section~\ref{sn:more} we provide conditions on $q$ and $w$ under which $0 \le \pi \le 1$. Finally, on a side note which may be of independent interest, in Section~\ref{sn:buildconvsequ} we discuss how our approach can be used to provide discrete approximations of continuous distributions under certain constraints.   In Section~\ref{sn:applications} we propose detailed computations in a variety of examples: exponential approximation in Section \ref{sn:expapp} (this includes Example \ref{ex:2}), gamma approximation in Section \ref{sec:gamma}, beta approximation in Section \ref{sn:beta} (this includes Example \ref{ex:1}),  normal approximation in Section \ref{sn:normalapprox} (this includes Example \ref{ex:3}), and asymptotics of queuing systems in Section \ref{sn:queue}. Many proofs are given in the text, but the less illustrative ones are relegated to  Appendix \ref{sec:proofs}.

\section{Stein's method with bespoke derivatives}
\label{sec:2}

Here and throughout, we assume that  $X \sim p, Z\sim q$ satisfy our Assumption; we also instate the notations from this Assumption. This will not be repeated. 

\subsection{Notations}
\label{sec:notations}

Let $\mathcal X := \{x_i, i \in I\}$; recall that this is an ordered set. Let $\mesh^-_i=\mesh^-(x_i) = x_{i}-x_{i-1}$ for all $i \in I\setminus \{0\}$, and $\mesh^+_i=\mesh^+(x_i) = x_{i+1}-x_i$  for all $i \in I$ (or $I \setminus \{\ell\}$ when $\ell \in \mathbb N$).  
Let $\mathcal X^\star$ be the collection of all  functions $f : \mathcal X \to \mathbb R$, which  we identify with the collection of column vectors  $f =(f_i)_{i \in I} = (f(x_i))_{i \in I}$ (the vector is an  infinite sequence  if $I = \mathbb{N}$).  For $f, g$ two elements of $\mathcal X^\star$ we also define $fg$ as the column vector with components $(fg)_i = f(x_i) g(x_i)$, and likewise for $f+g, f/g$ etc.   Finally, any function $f: \R \to \R$ is associated uniquely with a sequence $f\in \mathcal X^\star$ (we use the same notation for the function and the sequence) defined by $f_i = f(x_i)$ for $i \in I$. 
%
 %
  The   forward and backward derivatives are  then  defined on $\mathcal X^\star$ as 
\begin{align*}
    & (\Delta^+ f)_i = \Delta^+ f(x_i) = \frac{f_{i+1} - f_i}{\mesh^+_i}   \mbox{ and } (\Delta^- f)_i = \Delta^- f(x_i) = \frac{f_{i} - f_{i-1}}{\mesh^-_i}  \mbox{ for } i \in I
\end{align*}
with the convention that $(\Delta^- f)_0= \Delta^- f(x_0) = 0$, and  $(\Delta^+ f)_\ell = \Delta^+ f(x_\ell) = 0$ if $\ell < +\infty$.  
    For any sequence   $\pi = (\pi_i)_{i \in I} \in \mathcal X^\star$ we introduce    the weighted  derivative operators  \begin{equation*}
        (\Delta^{\pi} f)_i = \pi_i (\Delta^+ f)_i + (1-\pi_i) (\Delta^-f)_i \mbox{ for } i \in I. 
    \end{equation*}
  It will be useful to consider the action of $\Delta^{\pi}$ on $\mathcal{X}^\star$ as matrix multiplication, with 
\begin{equation*}
\Delta^{\pi} = \big( \Delta^{\pi}_{i, j}\big)_{i, j \in I} = \begin{pmatrix}
-\frac{\pi_0}{\m_1} & \frac{\pi_0}{\m_1} &0&  \cdots & 0 & 0 \\
- \frac{1-\pi_1}{\m_1} & \frac{1-\pi_1}{\m_1} - \frac{\pi_1}{\m_2}& \frac{\pi_1}{\m_2} &  \ldots & 0 & 0 \\
0 & - \frac{1-\pi_2}{\m_2} & \frac{1-\pi_2}{\m_2} - \frac{\pi_2}{\m_3} &\ldots & 0 & 0 \\
\vdots & \vdots &  \vdots &\ddots & \vdots & \vdots \\
0 & 0 & 0 & \cdots & -\frac{1-\pi_\ell}{\m_\ell} & \frac{1-\pi_\ell}{\m_\ell} 
\end{pmatrix}
\end{equation*}
(the matrix is of infinite size when $I = \mathbb N$). 
Then 
$\Delta^{\pi}f$ is the column vector obtained by right-multiplying $ \Delta^\pi$  with the column vector $f$, i.e. it is the vector with components $(\Delta^\pi f)_j = \sum_{k \in I} \Delta^\pi_{j k} f_k$ (where each summation is of at most three terms) for $j \in I$.  We also write $(\Delta^\pi)^tf $ for the vector obtained by pre-multiplying $f$ with $(\Delta^\pi)^t$, the transpose of $\Delta^\pi$.

\subsection{Comparison of weighted  Stein operators}
\label{sec:2.3}
Consider the following linear  operators, which we call Stein operators.

\begin{definition}[Weighted  Stein operators] Let $w: (a, b) \to \R_0^+$ be a twice differentiable function and also let $w = (w_i)_{i \in I}$ the corresponding element of $\mathcal X^\star$. Let $s(x) = (w(x) q(x))'/q(x)$ for $x \in (a, b)$. 

\begin{enumerate}
    \item Let $\pi  \in \mathcal X^\star$. The $\pi, w$-Stein operator for $p$  is the linear operator sending $f \in \mathcal X^\star$ onto $(\mathcal T_{p, w, \pi}f) \in \mathcal X^\star$ with components 
    \begin{equation} \label{eq:steinopd}
    (\mathcal T_{p, w, \pi}f)_i =w_i  (\Delta^{\pi} f)_i-\frac{((\Delta^{\pi})^t(w  p ))_i}{p_i} f_i
\end{equation}
for all $i \in I$.
\item The $w$-density Stein operator for $q$ is  the linear operator sending differentiable functions $f$ onto 
\begin{equation} \label{eq:steinopc}
    \T_{q, w} f(x) 
    = w(x) f'(x) + s(x) f(x)
\end{equation}
 for $x \in (a, b)$. 
\end{enumerate}

\end{definition}

\begin{remark}
Weighted Stein operators for continuous distributions have already received a lot of attention and our operator \eqref{eq:steinopc} is studied e.g.\ in \cite{ley2017distances}. Our discrete operator \eqref{eq:steinopd} is connected to  \cite{yang2018goodness} (see their Theorem 3). 
\end{remark}

The following two technical results follow from standard arguments (see Appendix \ref{sec:proofs} for the proofs). 
\begin{proposition}[Stein classes] \label{prop:pstinope}
$\mbox{}$
Let $\pi, w \in \mathcal X^\star$ be such that  $w \in L^1(p)$ and $w \pi \in L^1(p)$ (i.e.\  $\E[|w(X)|] := \sum_{i\in I}  |w_i| p_i<\infty$ and $\E[|\pi(X)w(X)|] := \sum_{i\in I} |w_i \pi_i| p_i <\infty$). Then $\mathbb E [\mathcal T_{p, w,  \pi} f(X)] = 0$ for all {bounded}  $f \in  \mathcal X^\star$. 
\end{proposition}

\begin{proposition}[Stein equations] \label{prop:qsteinope}
Let $\mathrm{Lip}(1)$ be  the  collection of functions $h : \mathbb R \to \mathbb R$ such that  $h$ is Lipschitz with a.e.\ derivative bounded by 1. For each $h \in \mathrm{Lip}(1)$,   the function defined on $(a, b)$ by 
\begin{equation}\label{eq_coucou}
    f_h(x)=\frac{1}{p(x)w(x)}\int_a^x (h(y)-\E[h(Z)])p(y) dy
\end{equation}
is solution to  $\mathcal T_{q, w} f_h (x)= h(x) - \mathbb E[h(Z)]$ for all $ x \in (a, b)$. Moreover, if {$s\in L^1(q)$ is strictly decreasing, continuous, if  $\E[s(Z)]=0$} and if $\tau_q/w$ is bounded on $(a, b)$ (recall the function $\tau_q$ from \eqref{steinkk}) then  $f_h$ is also bounded on $(a, b)$; in particular $\mathbb E [\mathcal T_{p, w, \pi} f_h(X)] = 0$ for all $h \in \mathrm{Lip}(1)$. 
\end{proposition}
 Propositions \ref{prop:pstinope} and \ref{prop:qsteinope} are all we need to perform Stein's method of comparison of operators, through our next result.
 
 \begin{theorem}\label{th:steincomp} Under the previous notations and assumptions it holds that     \begin{equation*}
    W_1(X, Z) = \sup_{h \in \mathcal H} \left|  \E\left[ w(X) \left(f_h'(X)-\Delta^{\pi} f_h(X)\right)+\left(s(X)+\frac{(\Delta^{\pi})^t (w  p )}{p}(X)\right)f_h(X)\right]\right|.
\end{equation*}
 \end{theorem}
 
 \begin{proof}
    First recall the classical representation 
     for the Wasserstein distance:
\begin{equation}\label{eqwa}
    W_1(X, Z) = \sup_{h \in \mathrm{Lip}(1)} |\E[h(X)]-\E[h(Z)]|.
\end{equation} 
Next  fix    $h\in \mathrm{Lip}(1)$  and  let $f_h$ be as  given in Proposition \ref{prop:qsteinope}. Then 
\begin{align*}
\E[h(X)]-\E[h(Z)]&=\E\left[\mathcal T_{q, w}  f_h(X) \right]  = \E\left[\mathcal T_{q, w}  f_h(X) \right] - \E\left[\mathcal T_{p, w, \pi}f_h(X) \right]\\
& = \E\left[(\mathcal T_{q, w}   - \mathcal T_{p, w, \pi})f_h(X) \right]. 
\end{align*}
Unravelling the expressions for the operators we note that  
\begin{equation*}
   (\mathcal T_{q, w}   - \mathcal T_{p, w,  \pi})f  
 = w \left(f'-\Delta^{\pi} f\right)+\left(s+\frac{(\Delta^{\pi})^t (w  p )}{p}\right)f
\end{equation*}
which gives the claim. 
 \end{proof}

\subsection{Bespoke weighted derivatives} \label{sec:besp}
 We suggest to tackle the approximation of $q$ by $p$  (or vice-versa) by   fixing  $w$ then choosing $\pi$ so  that 
\begin{equation}
\label{eq_deltas}
\frac{((\Delta^{\pi})^t (w  p ))_i}{p_i}=-s_i, \mbox{ for all } i \in I
\end{equation}
(still with $s_i = s(x_i)$). 
If \eqref{eq_deltas} has a solution $\pi(= \pi(w, p, q))$ then from Theorem \ref{th:steincomp} all that remains is to bound the expression
\begin{equation}\label{eq:dDelta}
 \sup_{h \in \mathcal H}  \left| \E\left[w(X) \left(f_h'(X)-\Delta^{\pi} f_h(X)\right)\right]\right|
\end{equation}
with $f_h$ as given in Proposition \ref{prop:qsteinope}. 

In order to apply our method, it  is first necessary to determine conditions under which equation \eqref{eq_deltas} admits a solution $\pi$ such that $\pi w\in L^1(p)$. We begin by noting  how, under the assumptions of Proposition \ref{prop:pstinope},   $\pi \in \mathcal X^\star$ solves \eqref{eq_deltas} if and only if 
\begin{equation}
\label{eq stein a}
\E\left[w(X)\Delta^{\pi} f(X) \right]=-\E\left[s(X) f(X)\right]
\end{equation}
for all {bounded function} $ f\in \mathcal{X}^\star$. This gives an interpretation to the  weights $\pi$, at least  when $w = \tau_q$ is the Stein kernel of $Z$ given by \eqref{steinkk}. Indeed it then follows that   $s(x) = -(x - \mathbb{E}[{Z}])$. Since  the Stein kernel also satisfies  the integration by parts  formula 
\begin{equation}\label{steinkkibp}
    \mathbb{E}[\tau_q(Z) f'(Z)] = \mathbb{E}[(Z - \mathbb{E}[Z]) f(Z)]
\end{equation}
(supposed to hold for all smooth $f$ with at most polynomial growth),  direct comparison of   \eqref{eq stein a}    and \eqref{steinkkibp} indicates that the coefficients $\pi$ allow to tweak the discrete derivatives $\Delta^{\pi}$ in such a way that  the Stein kernel of $X$ associated to $\Delta^{\pi}$ is equal  to the Stein kernel of $Z$ (associated to the usual derivative). This is why we call derivative operator $\Delta^\pi$ a bespoke derivative. 

Our next result (see  Appendix \ref{sec:proofs} for the proof)  provides sufficient conditions for existence of solutions $\pi$ to \eqref{eq_deltas} and also gives an explicit formula for the solutions, leading to formula \eqref{eq_ai3-intro}.

\begin{proposition}
\label{prop a}
If $I$ is finite the weights $(\pi_i)_{i \in I}$ defined by  \begin{align}
\label{eq_ai3}
\pi_i&=\frac{1}{p_i w_i}\sum_{j=0}^i \left((x_j-x_i)s_j+w_j\right)p_j, \quad i \in I
\end{align}
satisfy \eqref{eq_deltas} if and only if 
\begin{align}
\label{condition1}
\E[s(X)]=0\quad\text{and}\quad \E\left[X s(X) + w(X) \right]=0.
\end{align}
If $I=\N$ the same weights $(\pi_i)_{i \in I}$  satisfy \eqref{eq_deltas}, and $\pi w \in L^1(p)$ {only} if \eqref{condition1} holds.

\end{proposition}

\begin{remark}
    When $I = \N$,   $\pi w \in L^1(p)$ if  $$\lim_{i\to\infty}(\pi_{i+1}w_{i+1}p_{i+1})/(\pi_{i}w_{i}p_{i}) < 1$$
by the ratio test. 
\end{remark}
\begin{remark}
    The sequence from \eqref{eq_ai3} can be defined through the recurrence 
$\pi_0 = 1$, 
\begin{align*}
    \pi_{i} w_{i}p_{i} =  \pi_{i-1} w_{i-1} p_{i-1} +w_{i}p_{i}- (x_{i}-x_{i-1}) \sum_{j=0}^{i-1} s_j p_j, \quad i \in I\setminus\left\{0\right\}.
\end{align*}

\end{remark}
If we take $w=\tau_q$,  the conditions in  \eqref{condition1} read as 
\begin{equation}
\label{eq:w=tau}
\E[X]=\E[Z]\quad\text{and}\quad\Var[X]=\E[\tau_q(X)].
\end{equation}
This  can be further simplified when $q$ is an integrated Pearson density. Recall (see e.g.\ \cite{Afendras}) that $q$ belongs to the integrated Pearson family if  it has finite mean $m$  and its Stein kernel is  of the form $\tau_q(x)=\alpha x^2+\beta x+\gamma$.  We write $Z\sim \mathrm{IP}(m; \alpha, \beta, \gamma)$. In this case we obtain the following. 
\begin{proposition}
\label{w=tau2}
Let  $Z\sim \mathrm{IP}(m;\alpha, \beta, \gamma)$ and  take $w=\tau_q$ in Proposition \ref{prop a}. Then the conditions in \eqref{condition1} become 
\begin{equation}
\E[X]=\E[Z]\quad\text{and}\quad\Var[X]=\Var[Z].
\end{equation}
\end{proposition}
Finally, we note that if  $Y$ is any discrete random variable and  $q$ is integrated Pearson then the random variable 
\begin{equation}\label{eq:autosta}
    X=\sqrt{\Var[Z]}\frac{Y-\E[Y]}{\sqrt{\Var[Y]}} +\E[Z]
\end{equation}
verifies  \eqref{condition1}.

\subsection{The weights when $Z$ is standard Gaussian} 
\label{ex:xomepisnorm} We take $q= \varphi$  the standard normal density and $w \equiv 1$.  Then \eqref{condition1} imposes $\mathbb{E}[X] = 0$ and $\mathrm{Var}[X] = 1$, and if furthermore $X = (Y-\mu)/\sigma $ with $Y$ integer valued, {$\mu=\E[Y]$ and $\sigma^2=\Var[Y]$} then \eqref{eq_ai3} becomes 
\begin{align}
\label{eq_ai3gau}
\pi_i&=\frac{1}{p_i }\sum_{j=0}^i \left(1 - \frac{j-\mu}{\sigma}\frac{j-i}{\sigma}\right)p_j, \quad i \in I. 
\end{align}
The following results hold (see the supplementary material for computations). 
\begin{enumerate}
    \item 
Let $Y \sim \mathrm{Bin}(n, t)$ for $n \in \mathbb N$ and $t \in (0, 1)$. Condition \eqref{condition1} is satisfied for  $X = (Y - nt)/\sqrt{nt(1-t)}$ so that  $\ell = n$, and $\mathcal{X} = \{ x_i = (i-nt)/\sqrt{nt(1-t)}, i = 0, \ldots, n\}$. Identity \eqref{eq_ai3gau} yields   
$$\pi_i=\frac{n-i}{n}$$ for $i = 0, \ldots, n. $ Note how $\pi_i \in [0, 1]$ for all $i = 0, \ldots, n$. Derivative  $\Delta^\pi$ was already studied, in the context of the binomial distribution, in  \cite{hillion2014natural}.
\item If  $Y \sim \mathrm{Poi}(\lambda)$ for $\lambda >0$ then we take $X = (Y - \lambda)/\sqrt{\lambda}$  so that $\ell = +\infty$, and  $\mathcal X = \{x_i=(i-\lambda)/\sqrt{\lambda}, i = 0, 1, \ldots \}$. Identity \eqref{eq_ai3gau} yields   $$\pi_i=1$$
for all $i\in\N$ so that $\Delta^\pi$ is simply the forward derivative in this case. 
\item If $Y$ follows  the negative binomial density with parameters $n$ and $t$ then we take $X = (Y - n(1-t)/t)/\sqrt{n(1-t)/t^2}$ so that $I = \mathbb{N}$, and $\mathcal X = \{x_i=(i-n(1-t)/t)/\sqrt{n(1-t)/t^2}, i = 0, 1, \ldots \}$. Identity \eqref{eq_ai3gau} yields    $$\pi_i=1+\frac{i}{n},$$
for all $i \ge 0$.
Note how $\pi_i \ge 1$ for all $i \ge 0$.
\item If $W$ follows  the hypergeometric  density with integer parameters $r, n, N$ such that $r\geq n$ and $N>n+r$ so that 
$
p_i= {\binom{n}{i} \binom{N-n}{r-i}}/{\binom{N}{r}},$ $ i = 0, \ldots, n$
then $\mu =nr/N $ and $\sigma^2= {n(1-n/N)(N-r)r/(N(N-1))}$ we take 
$X = (Y - \mu)/ \sigma$ so that $\ell = n$, and $\mathcal X = \{x_i=(i - \mu)/ \sigma, i = 0, 1, \ldots, n\}$. Identity \eqref{eq_ai3gau} yields    
\begin{equation*}
    \label{hypergeomai}
    \pi_i = \frac{(n-i)(r-i)(N i + (N-n)(N-r))}{nr(N-n)(N-r)}
\end{equation*}
for all $ i = 0, \ldots, n.$
Note that  $0 \le \pi_i \le 1$ for all $i = 0, \ldots, n$.
\item 
If $Y$ follows  the  uniform distribution on $\left\{0,\ldots,n\right\}$ then $\mu = n/2$ and $\sigma^2 = ((n+1)^2-1)/12$. We take $X = (Y - \mu)/\sigma$ so that $\ell =  n$, and $\mathcal X = \{x_i=(i-n/2)/\sqrt{((n+1)^2-1)/12}, i = 0, 1, \ldots, n \}$. Identity \eqref{eq_ai3gau} yields   $$\pi_i=\frac{(i+1)(n-i)(n-2(i-1))}{n(n+2)}$$ for $i = 0, \ldots, n$. These are neither necessarily bounded between 0 and 1 nor even positive. 
\end{enumerate}

Finally, we show that the coefficients $(\pi_i)_{i \in I}$ can  be explicitly  related to the third moment of $X$.
\begin{proposition} \label{prop:gaussthirdmom}
Let  $Y\sim p$ be a random variable on $\left\{0,\ldots,\ell\right\}$ ($\ell\in\N$) with mean $\mu$ and variance $\sigma^2$, and  let $\pi$ be given by \eqref{eq_ai3gau}. Then 
$$
\E\left[\pi\left(Y\right)\right]=\frac{1}{2}\left(1+\frac{1}{\sigma^2}\E\left[(Y-\mu)^3\right]\right).
$$
 
\end{proposition}

\begin{remark}
 There are some similarities between our coefficients $\pi_i$ and the conditional probabilities $\beta_i$ appearing in \cite{holmes2004stein}; we have not been able to exploit this  connection in our applications at this stage. 

\end{remark}

\subsection{About the Stein  factors}
\label{sn:abstractwass}
We now return to the general setting to define and study the constants $C_q^0(w)$ and  $C_q^1(w)$  under the assumptions  stated in  Theorem \ref{theo main intro}.  

\begin{definition}\label{defn:2.10}
    For each $h \in \mathrm{Lip}(1)$ let $f_h$  be  given by \eqref{eq_coucou}. The Stein  factors  for $q$ with weight $w$ are
    \begin{equation}
        \label{eq:stima}
        C_q^0(w)=\frac{1}{2}\left(\sup_{h \in \mathrm{Lip}(1)} \| w'f_h'\|_\infty+\sup_{h \in \mathrm{Lip}(1)}\| (f_h'w)'\|_\infty\right),
    \end{equation}
    \begin{equation*}
         C^1_q(w)  = \sup_{h \in \mathrm{Lip}(1)} \frac{1}{6}\| f_h'\|_\infty \|w''\|_{\infty}
    \end{equation*}
    where $\| \cdot\|_\infty$ is the essential supremum over $(a, b)$. 
\end{definition}  
 Inspired by  \cite{dobler2015stein} we readily obtain  the following result.   
 \begin{lemma}
\label{prop borne}
 Let $F$ be the cumulative distribution function of $q$ and $\bar F $ the corresponding  survival function,  and set
\begin{align}
\label{eq Gamma12}
\Gamma_1(x)=\int_a^x  F
\quad\text{and}\quad \Gamma_2(x)=\int_x^b \bar{F} 
\end{align}
for all $x \in (a, b)$. 
Suppose that $w$ and $s$ satisfy the assumptions from Theorem \ref{theo main intro}. 
Then, for all $h\in Lip(1)$,
 \begin{align}
 |f'_h| & \leq  \left(\frac{1}{w}+\frac{s\bar{F}}{qw^2} \right)\Gamma_1
+\left(\frac{1}{w}- \frac{sF}{qw^2}\right)\Gamma_2  \label{eq:boundfp}\\ 
|(f_h'w)'| & \leq   \left|\frac{s}{w}+\frac{s^2\bar{F}}{qw^2}-\frac{s'\bar{F}}{w q}\right|\Gamma_1
+\left|\frac{s}{w}-\frac{s^2F}{qw^2}+\frac{s'F}{w q}\right|\Gamma_2+1 \label{eq:boundfwp}
 \end{align}
 \end{lemma}
 
Recall the Stein kernel $\tau_q$ defined in \eqref{steinkk}. Among the many possible choices,  $w=\tau_q$ (and thus $s(x) = \E[Z]- x$) appears to be  most natural for many target distributions $q$. 
The Stein factors can be recomputed for this choice of weight, as follows. 

\begin{lemma}
\label{prop borne 2} Instate the notations of Lemma \ref{prop borne}, take $w=\tau_q$ and write $\overline{\mathrm{Id}}(x) = \E[Z]-x$. Then, 
\begin{align}
\label{eq Gamma12}
\Gamma_1=q\tau_q-\overline{\mathrm{Id}}F
\quad\text{and}\quad \Gamma_2 = q\tau_q+\overline{\mathrm{Id}}\bar{F}
\end{align}
and, for all $h\in\text{Lip}(1)$, 
\begin{align*}
|f_h'|& \leq 2 \frac{\Gamma_1 \Gamma_2}{q \tau_q^2}, \\
 |(f'_h\tau_q)'| & \leq  \left|\frac{\overline{\mathrm{Id}}}{\tau_q}+\frac{\overline{\mathrm{Id}}^2\bar{F}}{q\tau_q^2}+\frac{\bar{F}}{q \tau_q}\right|\Gamma_1
+\left|1-\left(\frac{\overline{\mathrm{Id}}}{\tau_q}+\frac{\overline{\mathrm{Id}}^2\bar{F}}{q\tau_q^2}+\frac{\bar{F}}{q\tau_q}\right)\Gamma_1\right|+1.
\end{align*}
Moreover, we have $|(f_h'\tau_q)'|\leq 2$ for all $h\in \mathrm{Lip}(1)$ if 
\begin{equation*}
\frac{\overline{\mathrm{Id}}q\tau_q}{\overline{\mathrm{Id}}^2+\tau_q}\leq F\leq \frac{\overline{\mathrm{Id}}q\tau_q}{\overline{\mathrm{Id}}^2+\tau_q}+1.
\end{equation*}
This last condition is satisfied e.g.\  when $-2\tau_q\leq \overline{\mathrm{Id}} \tau_q'$. 
\end{lemma}

The above partially simplifies when $Z$ belongs to the integrated Pearson family.
    
\begin{proposition}
\label{prop borne3}
Let $Z\sim \mathrm{IP}(m;\alpha, \beta, \gamma)$. If its support is a finite or a semi-finite interval, we have 
$||(f_h'\tau_q)'||_{\infty}\leq 2$ for all $h\in \mathrm{Lip}(1)$. 
If its support is $\R$, we have $||(f_h'\tau_q)'||_{\infty}\leq 2$ for all $h\in \mathrm{Lip}(1)$ if  $\beta=-2\E[Z]\alpha$ and $ \E[Z]\leq \sqrt{\gamma/\alpha}$. In particular, it is the case if $Z\sim q$ follows a Gaussian or a Student distribution. 
\end{proposition}

\begin{remark}
      Lemma \ref{prop borne 2} and Corollary \ref{prop borne3} give the universal bound $||(f_h'\tau_q)'||_{\infty}\leq 2$  for all $h\in \mathrm{Lip}(1)$ for many  cases. This result seems to be new.  We have  not been able to determine such a bound for $||f_h'\tau_q'||_{\infty}$,  though it is easy to bound this quantity on a case-by-case basis.  Moreover the literature also contains bounds on $||f_h'||_{\infty}$ for many interesting settings which we can use directly. 
\end{remark}

 \subsection{The bounds and their proofs} \label{sn:proof}
We begin by  proving Theorem \ref{theo main intro}.
\begin{proof}[Proof of Theorem \ref{theo main intro}]
We suppose that $I$ is finite; the case $I = \mathbb{N}$ follows along the exact same lines. 
For fixed $w$ satisfying our assumptions, recalling \eqref{eq:dDelta}, our purpose is to bound $\left| \E\left[w(X) \left(f_h'(X)-\Delta^{\pi} f_h(X)\right)\right]\right|$
uniformly in $h\in \mathrm{Lip}(1)$ with $f_h$ given by \eqref{eq_coucou}. Now fix $h\in \mathrm{Lip}(1)$.  Writing  $f$ instead of $f_h$ we have 
\begin{align*}
\left| \E\left[w(X) \left(f'(X)-\Delta^{\pi} f(X)\right)\right]\right|
& \leq  \sum_{i=0}^\ell \left| \pi_i \left(f'(x_i)-\Delta_+ f(x_i)\right)w(x_i)p_i \right| \\
& +\sum_{i=0}^\ell \left| (1-\pi_i)\left(f'(x_i)-\Delta_{-}f(x_i)\right) w(x_i)p_i\right| \\
& =: L_1+L_2.
\end{align*}
We only prove the bound on $L_1$, since the bound on $L_2$ follows by the exact same computation.  First, 
\begin{align*}
L_1& =\sum_{i=0}^\ell \left|\frac{\pi_i}{\mesh_{i}^+} \int_{x_i}^{x_{i+1}}\left(f'(x_i)-f'(y)\right)dy\  w(x_i)p_i\right|\\
&=\sum_{i=0}^\ell \left|\frac{\pi_i}{\mesh_{i}^+} \int_{x_i}^{x_{i+1}}\left[f'(x_i)w(x_i)-f'(y)w(y)+(w(y)-w(x_i))f'(y)\right]dy\,  p_i\right|\\
&=\sum_{i=0}^\ell\left| \frac{\pi_i}{\mesh_{i}^+} \int_{x_i}^{x_{i+1}}\int_{x_i}^y\left[-(f'(t)w(t))'+w'(t)f'(y)\right]dt\, dy\,  p_i\right|\\
&=\sum_{i=0}^\ell\left| \frac{\pi_i}{\mesh_{i}^+} \int_{x_i}^{x_{i+1}}\int_{x_i}^y\left[-(f'(t)w(t))'+w'(y)f'(y)+(w'(t)-w'(y))f'(y)\right]dt\,dy\, p_i\right|\end{align*}
Uniformly bounding the innermost expression  we get  the bound 
\begin{align*}
L_1
&\leq \sum_{i=0}^\ell \frac{|\pi_i|}{\mesh_{i}^+} \int_{x_i}^{x_{i+1}}\int_{x_i}^y\left[||(f'w)'||_{\infty}+||w'f'||_{\infty}+||w''||_{\infty}||f'||_{\infty}(y-t)\right]dt\,dy\, p_i\\
&=\sum_{i=0}^\ell |\pi_i|\left[\frac{\mesh_{i}^+}{2} \left(||(f'w)'||_{\infty}+||w'f'||_{\infty}\right)+\frac{(\mesh_{i}^+)^2}{6}||w''||_{\infty}||f'||_{\infty}\right]  p_i.
\end{align*}
This gives our result. \end{proof}

Obviously we made several editorial choices during the above proof, and many variations are possible depending on the specific context at hand. For example whenever $C_q^1(w) \neq 0$ 
it may be handier to use the following result which can be obtained through only minor modifications of the above proof. 

\begin{theorem}\label{theo2dform}
    Under the same assumptions it holds that 
\begin{equation*}
    W_1(X,Z)\leq C_q(w)\mathbb{E}\left[|\pi(X)| \mesh^+(X) + |1-\pi(X)| \mesh^-(X) \right]
\end{equation*}
where $$C_q(w)=1/2\left(\sup_{h \in \mathrm{Lip}(1)} \| f_h'\|_\infty \| w'\|_\infty+\sup_{h \in \mathrm{Lip}(1)}\| (f_h'w)'\|_\infty\right).$$ 
\end{theorem}

We will use this result in Section \ref{sn:beta}, in the context of beta approximation. 

\begin{remark}
Of course it may be the case that $\|w''\|_\infty = \infty$. 
If $w$ is a polynomial of order $k$, we propose to  use the following generalization of the previous argument
\begin{align*}
W_1(X,Z)\leq &C_q^0(w)  \, \mathbb{E}\left[|\pi(X)| \mesh^+(X) + |1-\pi(X)| \mesh^-(X) \right]\\
&+\sum_{j=2}^{\ell}\frac{1}{(j+1)!}\sup_{h \in \mathrm{Lip}(1)} \| w^{(j)}f_h'\|_\infty  \, \mathbb{E}\left[|\pi(X)| (\mesh^+(X))^{j} + |1-\pi(X)| (\mesh^-(X))^j \right].
\end{align*}
We have not studied any example where this came in useful, hence do not pursue this further here.  
\end{remark}

It may occur that uniformly bounding on $|(f_h'w)'|$, $|f_h'w'|$, $|f_h'|$ and $|w''|$ as done  above  does not yield interesting bounds. The following variation of the previous result may then be used. 

\begin{theorem}
    \label{rem:notroutine} 
 Still under the same    assumptions it holds that 
\begin{align*}
W_1(X,Z)\leq &  \sum_{i=0}^\ell  \left(C_q^{0,i+1}(w)|\pi_i|  \mesh_{i}^+ + C_q^{0,i}(w)|1-\pi_i|  \mesh_{i}^- \right) p_i\\
&+\sum_{i=0}^\ell  \left(C^{1,i+1}_q(w)|\pi_i|  (\mesh_{i}^+)^2+ C^{1,i}_q(w)|1-\pi_i|  (\mesh_{i}^-)^2\right) p_i
\end{align*}
where $C_q^{0,i}(w)=1/2\sup_{h \in \mathrm{Lip}(1)} \sup_{[x_{i-1},x_{i}]}| (f_h'w)'| +1/2\sup_{h \in \mathrm{Lip}(1)}\sup_{[x_{i-1},x_{i}]} | w'f_h'| $ and $C_q^{1,i}(w)=1/6\sup_{h \in \mathrm{Lip}(1)} \sup_{[x_{i-1},x_{i}]}| f_h'|\sup_{[x_{i-1},x_{i}]}| w''|$ for $i\in I\setminus\left\{0\right\}$ and $C_q^{0,0}(w)=0=C_q^{1,0}(w)$ and $C_q^{0,\ell+1}(w)=0=C_q^{1,\ell+1}(w)$ if $\ell< \infty$. 
     
\end{theorem}
 We will use this result in Section \ref{sn:queue}, in the context of convergence to equilibrium of classical queueing models.

\subsection{More about the weights}\label{sn:more}
In the sequel, we will use the notation $\delta_i=\delta_i^-=x_i-x_{i-1}$. 
Whenever it is not possible to obtain explicit expressions for the weights $\pi_i$,   the following results may be useful. 
\begin{proposition}
\label{prop borne sup}
Assume that $I$ is finite (i.e.\ $\ell < \infty$).
To ensure that the coefficients $(\pi_i)_{i\in I}$ defined in \eqref{eq_ai3} satisfy $\pi_i\leq 1$ for all $i\in I$, it is sufficient that one of the following conditions holds.
\begin{enumerate}
\item There exists $i_1\in \left\{0,\ldots,\ell-1\right\}$ such that $p_{i}w_{i}- \delta_{i+1}\sum_{j=0}^is_jp_j$ is negative for $i=0,\ldots,i_1$ and positive for $i=i_1+1,\ldots,\ell-1$.
\item $w_0\leq \delta_1 s_0$, $w_{\ell-1}p_{\ell-1}\leq -\delta_{\ell}s_{\ell}p_{\ell}$ and there exists $i_2\in\left\{0,\ldots,\ell-1\right\}$ such that $p_{i}w_{i}/\delta_{i+1}-p_{i-1} w_{i-1}/\delta_{i}- s_{i}p_{i}$ is negative for $i=1,\ldots,i_2$ and positive for $i=i_2+1,\ldots,\ell-1$.
\item $w_0\leq \delta_1 s_0$, $w_{\ell-1}p_{\ell-1}\geq -\delta_{\ell}s_{\ell}p_{\ell}$ and there exists $i_2,i_3\in\left\{0,\ldots,\ell-1\right\}$ such that $i_2\leq i_3$, $p_{i}w_{i}/\delta_{i+1}-p_{i-1} w_{i-1}/\delta_{i}- s_{i}p_{i}$ is negative for $i=1,\ldots,i_2$ and $i=i_3+1,\ldots,\ell-1$ and positive for $i=i_2+1,\ldots,i_3$.
\end{enumerate}
\end{proposition}
We have similar conditions for the lower bound.
\begin{proposition}
\label{prop borne inf}
Assume that $I$ is finite (i.e.\ $\ell < \infty$).
To ensure that the coefficients $(\pi_i)_{i\in I}$ defined in \eqref{eq_ai3} satisfy $\pi_i\geq 0$ for all $i\in I$, it is sufficient that one of the following conditions holds.
\begin{enumerate}
\item There exists $i_1\in \left\{0,\ldots,\ell-1\right\}$ such that $p_{i}w_{i}- \delta_{i}\sum_{j=0}^{i-1}s_jp_j$ is positive for $i=1,\ldots,i_1$ and negative for $i=i_1+1,\ldots,\ell$. 
\item $w_{\ell}\leq -\delta_{\ell} s_{\ell}$, $w_{1}p_{1}\leq \delta_{1}s_{0}p_{0}$ and there exists $i_2\in\left\{0,\ldots,\ell-1\right\}$ such that \linebreak$p_{i+1}w_{i+1}/\delta_{i+1}-p_{i} w_{i}/\delta_{i}- s_{i}p_{i}$ is negative for for $i=1,\ldots,i_2$ and positive for $i=i_2+1,\ldots,\ell-1$. 
\item $w_{\ell}\leq -\delta_{\ell} s_{\ell}$, $w_{1}p_{1}\geq \delta_{1}s_{0}p_{0}$ and there exists $i_2,i_3\in\left\{0,\ldots,\ell-1\right\}$ such that $i_2\leq i_3$ and $p_{i+1}w_{i+1}/\delta_{i+1}-p_{i} w_{i}/\delta_{i}- s_{i}p_{i}$ is positive for $i=1,\ldots,i_2$ and $i=i_3+1,\ldots,\ell-1$ and negative for $i=i_2+1,\ldots,i_3$. 
\end{enumerate}
\end{proposition}
It is easy to show that $w_0\leq \delta_1 s_0$ and $w_{\ell}\leq -\delta_{\ell} s_{\ell}$ are necessary conditions to have $0\leq\pi_i\leq 1$ for all $i=0,\ldots,\ell$. If $Z$ is standard Gaussian and $X = (Y-\E[Y])/\sqrt{\Var[Y]}$ with $Y$ integer valued (as in Section \ref{ex:xomepisnorm}) with support $\{0, \ldots, \ell\}$, these last two inequalities can be rewritten as 
\begin{equation}
\label{necessary gaus}
\Var[Y]\leq \min\left\{\E[Y],\ell-\E[Y]\right\}.
\end{equation}



\subsection{Building converging sequences}\label{sn:buildconvsequ}
In this final section we consider the following variation on the question which animated us thus far: given a continuous variable $Z$ and an ordered set of points $\left\{x_i, i\in I\right\}$, build a family of discrete variables $X_{\mesh}$ such that $W_1(X_{\mesh},Z)\leq C \mesh$
for some constant $C>0$ where $\mesh=\sup_{i \in I\setminus\left\{0\right\}}(x_{i}-x_{i-1})$. An answer is provided in our next result. 
\begin{proposition}
\label{prop panjer}
Let $Z\sim q$ and   $w : (a, b) \to \mathbb R^+_0$ be two times differentiable. 
Let  $\tau_q$ be the Stein kernel of $q$ and suppose that  $\tau_q/w$ is bounded on $(a, b)$,     $s\in L^1(q)$, $s$ is strictly decreasing and continuous, and  $\E[s(Z)]=0$.
Let $\left\{x_i, i\in I\right\}\subset \overline{(a,b)}$ be a set of ordered points, with $I=\left\{0,\ldots,\ell\right\}$ or $I=\N$, such that 
$\delta_1 s_0\geq w_0$, $w_i+\delta_i s_i\geq 0$ for all $i\in I\setminus \left\{0\right\}$ and $w_{\ell}=-\delta_{\ell}s_{\ell}$ if $I$ is finite. 
Define $X$ as a discrete variable on $\left\{x_i, i\in I\right\}$ with density $p=(p_i)_{i\in I}$ given by $p_1 w_1=\left(\delta_1 s_0-w_0\right)p_0$, $p_2w_2\delta_1/\delta_2=\left(w_1+\delta_1s_1\right)p_1+w_0p_0$ and
\begin{equation*}
p_{i}= \frac{\delta_{i}}{w_{i}}\left(\frac{w_{i-1}}{\delta_{i-1}}+s_{i-1}\right)p_{i-1},\quad i\in I\setminus \left\{0,1,2\right\}
\end{equation*}
Then, if $w\in L^1(p)$, we have
\begin{equation*}
W_1(X,Z)\leq C_q^0(w)\left(\delta_1p_0+\sum_{i=1}^{\ell}\delta_i p_i\right)+C_q^1(w)\left((\delta_1)^2p_0+\sum_{i=1}^{\ell}(\delta_i)^2p_i\right)
\end{equation*}
where $\ell=\infty$ if $I=\N$.
\end{proposition}
\begin{proof}
In order to use Theorem \ref{theo main intro}, we have to check that equality \eqref{eq_deltas} holds with $\pi_0=1$ and $\pi_i=0$ for all $i\in I\setminus \left\{0\right\}$. 
The first equation of \eqref{eq_deltas} is 
$
-p_1 w_1=\left(-\delta_1s_0+w_0\right)p_0.
$
The following $\ell-1$ equations of \eqref{eq_deltas} are
$$
\frac{1}{\delta_1}w_{0}p_{0}+\frac{1}{\delta_1}w_1p_1-\frac{1}{\delta_{2}}w_{2}p_{2}=-s_1p_1
$$
and
$$
\frac{1}{\delta_{i-1}}w_{i-1}p_{i-1}-\frac{1}{\delta_{i}}w_{i}p_{i}=-s_{i-1}p_{i-1}
$$
for $\in I\setminus \left\{0,1,2\right\}$, which can be written as $p_2w_2\delta_1/\delta_2=\left(w_1+\delta_1s_1\right)p_1+w_0p_0$ and $p_{i}w_{i}= \delta_{i}\left(w_{i-1}/\delta_{i-1}+s_{i-1}\right)p_{i-1}$. If $I$ is finite, the last equation of \eqref{eq_deltas} is
$
w_{\ell}p_{\ell}/\delta_{\ell}=-s_{\ell}p_{\ell}
$
which holds by assumption. 
In order to have $p_i\geq 0$ for all $i\in I$, it is enough that $\delta_1 s_0\geq w_0$ and $w_i+\delta_is_i\geq 0$ for all $i\in I\setminus \left\{0\right\}$. 
Finally, if $I$ is infinite, we have $w\in L^1(p)$ by assumption and $\pi w \in L^1(p)$ is obvious.
\end{proof}

\begin{corollary}
\label{coro panjer}
Let $Z\sim \mathrm{IP}(m;\alpha,\beta,\gamma)$ with support  $\overline{(a,b)}$ with $a\in\R$. 
(i) If $b<\infty$, take $\delta>0$ and $\ell\in\N$ such that $a+\delta \ell<b$ and $\tau_q(a+\delta \ell)=-\delta (\E[Z]-a-\delta \ell)$ and set $I=\left\{0,\ldots,\ell\right\}$; (ii) if $b=\infty$ and $\alpha<1$, take $\delta\leq \inf_{(\E[Z],b)}-\tau_q/\overline{\mathrm{Id}}$ and set $I=\N$. In both cases, set $\mathcal X = \left\{a+\delta i,  i\in I\right\}$ and  
define $X$ as a discrete variable on $\mathcal X$ with density $p=(p_i)_{i\in I}$ given by
\begin{equation*}
\label{eq pi}
p_{i}= p_0\prod_{j=0}^{i-1}\frac{\tau_j+\delta \overline{\mathrm{Id}}_j}{\tau_{j+1}}.
\end{equation*}
Then
$
W_1(X,Z)\leq C_q^0(\tau_q)\delta+C_q^1(\tau_q)\delta^2. 
$
\end{corollary}

\section{Applications}
\label{sn:applications}

\subsection{Exponential approximation}
\label{sn:expapp}

Throughout this subsection, we let $Z \sim \mathrm{Exp}(\lambda)$ an   exponential variable  with parameter $\lambda>0$, so that $q(x) \propto e^{- \lambda x} \mathbb{I}[x \ge 0]$.  The exponential distribution belongs to the Integrated Pearson family and $\tau_q(x)=x/\lambda$. Therefore we  can apply Corollary~\ref{prop borne3} to get $\sup_{h\in \text{Lip}(1)}||(f_h'\tau_q)'||_{\infty}\leq 2$. Direct computations also  provide $\sup_{h\in \text{Lip}(1)}||\tau_q'f_h'||_{\infty}\leq 1 $ (see Lemma~\ref{boundlam} in the appendix for a proof) and thus the constant \eqref{eq:stima} reads as 
\begin{equation*}
    C_q^0(\tau_q) = \frac{3}{2} \mbox{ and } C_q^1(\tau_q)=0;
\end{equation*}
in particular this constant does not depend on $\lambda$.  Now let $X \sim p$ be a discrete random variable. In order to satisfy \eqref{condition1} for $w = \tau_q$   it is necessary that  $\mathbb E[X] = 1/\lambda$ and $\mathrm{Var}[X] = {1}/{\lambda^2}$.

\begin{proposition}

Let   $Y$ be  a geometric random variable with parameter $t \in (0,1)$ (defined on $0,1,2,\ldots $) and $Z\sim \mathrm{Exp}(\lambda)$. Let  
$$
X=\frac{t}{\lambda\sqrt{1-t}}\left(Y-\frac{1-t}{t}\right)+\frac{1}{\lambda}.
$$
Then 
\begin{equation}
    W_1(X,Z)\leq \frac{3t}{2\lambda\sqrt{1-t}}.
\label{ineqexp1}
\end{equation}
\end{proposition}
\begin{proof}
The standardization of $X$ follows from \eqref{eq:autosta}. 
After some computations, the sequence of weights $(\pi_i)_{i \in \mathbb N}$ from \eqref{eq_ai3} can be seen to be given by 
$$
\pi_i=(1+i)\frac{\sqrt{1-t}-(1-t)}{t\, i+\sqrt{1-t}-(1-t)}, \quad i \in \N.
$$
Since $\lim_{i\to \infty}\pi_{i+1}w_{i+1}p_{i+1}/(\pi_i w_i p_i)=1-t=\lim_{i\to \infty}w_{i+1}p_{i+1}/( w_i p_i)$, we have $\pi w\in L^1(p)$ and $w\in L^1(p)$ by the d'Alembert rule. All assumptions required for Theorem \ref{theo main intro} are seen to be satisfied. 
As $t\geq \sqrt{1-t}-(1-t)$, we have $\pi_i\in [0,1]$ for all  $t\in (0,1)$ and all $i\in \mathbb{N}$ so that  \eqref{simplll} yields the claim. 
\end{proof}
\begin{remark}
Let  $(Y_n)_{n\ge1}$ be a sequence of  geometric random variables  with parameter $t_n \in (0,1)$. It is well known that, given    any sequence  $(\alpha_n)_{n\ge 1}$ such that $\alpha_n\to 0$ and $t_n/\alpha_n\to \lambda$ it then holds that the sequence  $(\alpha_n Y_n)_{n\ge 1}$ converges in distribution to an exponential distribution with rate $\lambda$ when $n \to \infty$.  Combining  \eqref{ineqexp1}  with \eqref{rema1} we get  that
$$  \left|{\alpha_n}\frac{1-t_n}{t_n} -\frac{1}{\lambda} \right| \le   W_1(\alpha_n Y_n,Z)\leq 
\frac{3t_n}{2\lambda\sqrt{1-t_n}}+\frac{1-\sqrt{1-t_n}}{\lambda}+\left|\alpha_n\frac{1-t_n}{t_n}-\frac{\sqrt{1-t_n}}{\lambda}\right|$$  for all $t_n$ sufficiently small 
(the lower bound is simply  the difference of the means). 
Note that when $\lambda = 1$ and $\alpha_n = t_n/(1-t_n)$, e.g.\ \cite[Proposition 7]{boutsikas2002distance} yields  $W_1(\alpha_n Y_n, Z) \le t_n/(1-t_n)$; this last bound is better than ours in that specific case.

\end{remark}

As anticipated in the introduction, we now tackle an  example   from  \cite{chatterjee2011exponential, chatterjee2011nonnormal}. We shall prove the following. 

\begin{proposition}\label{prop:chatt}
Consider the uniform measure on the set of the
$\binom{n}{n/2}$
eigenvalues of the
Bernoulli-Laplace Markov chain appearing in proportion to their multiplicities. Let $Y$ be a random variable chosen from this distribution and $Z\sim \mathrm{Exp}(1)$. If we set $X=\frac{n}{2}Y+1$, we have 
\begin{equation*}
    W_1(X, Z)  \le \frac{3\sqrt2}{\sqrt{n}} + \frac{3}{n}.
\end{equation*}
\end{proposition}

\begin{proof}
We follow \cite[Section 4]{chatterjee2011exponential}, where it is explained that  $X$ is  the random variable with values in $\mathcal X = \{x_0, \ldots, x_\ell\}$ (here $\ell=n/2$) with  $x_i=i(i+1)/\ell$ for all $i = 0, 1, \ldots, \ell$ and  $p_i = (\binom{2\ell}{\ell-i} - \binom{2\ell}{\ell-i-1}) / \binom{2\ell}{\ell}$ for $0 \le i \le \ell$. Then $\mathbb E[X] = \mathrm{Var}[X] = 1$ so that no standardization is necessary when comparing $X$ to $Z$.  Direct computations yield the  coefficients $\pi_i$ from \eqref{eq_ai3}: 
\begin{equation*}
\pi_i =     \frac{(\ell-i)(i+1)}{\ell(2i+1)}, i = 0, \ldots, \ell. 
\end{equation*}
Clearly $\pi_0 = 1 \geq \pi_{1} \geq \cdots \geq \pi_{\ell-1}  \geq \pi
_\ell   = 0$. In order to apply  Theorem \ref{theo main intro} we note how 
\begin{equation*}
    0 \le \pi_i \delta^+_i + (1-\pi_i) \delta^-_i \le  2\frac{1+i}{\ell}
\end{equation*} 
for all $i = 0, \ldots, \ell$. Direct computations yield 
\begin{equation*}
    \frac{2}{\ell}\left(1+\sum_{i=0}^\ell i p_i\right) = \frac{4^\ell + \binom{2\ell}{\ell}}{\ell \binom{2\ell}{\ell}} \le  \frac{2}{\sqrt{\ell}} + \frac{1}{\ell}
\end{equation*}
(where we used the well-known bound $\binom{2\ell}{\ell} \ge 2^{2\ell-1}/\sqrt \ell$ on the denominator) so that, multiplying by $3/2$, we get the claim. 
\end{proof}

\begin{remark}
In \cite[Section 4]{chatterjee2011exponential}     the proposed bound is in Kolmogorov distance. The paper  \cite[Theorem 3.2]{chatterjee2011nonnormal} considers the same problem in Wassertein distance and gives, through a different application of Stein's method, a bound of the same order as ours but with  a constant equal to 12. Numerical explorations indicate that the value of the constant given by Theorem \ref{theo main intro} is, in fact, close to $3 \sqrt{\pi/2}$; we have not been able to prove this. 
\end{remark}

\subsection{Gamma approximation} 
\label{sec:gamma}
Throughout this subsection, we let $Z \sim \Gamma( \alpha, \beta)$ be a Gamma variable with parameters $\alpha$ and $\beta$ so that $q(x) \propto e^{-\beta x} x^{\alpha-1} \mathbb{I}[x>0]$. This is an Integrated Pearson random variable with Stein kernel  $\tau_q(x)=x/\beta$. We know that $\sup_{h\in \text{Lip}(1)}||(f_h'\tau_q)'||_{\infty}\leq 2$ and Theorem 2.1 in \cite{dobler2018gamma} (see also their Remark 2.2) ensures that $\sup_{h\in \text{Lip}(1)}||f_h'\tau_q'||_{\infty}\leq 2$. We have thus  
$$C^0_q(\tau_q)\leq 2 \mbox{ and }C_q^1(\tau_q) = 0$$
(the first  bound is not optimal and more effort could be invested to lower the value of the constant, but we do not pursue this here).
Now let $X \sim p$ be a discrete random variable. In order to satisfy \eqref{condition1} for $w = \tau_q$   it is necessary that  $\mathbb E[X] = \alpha/\beta$ and $\mathrm{Var}[X] = {\alpha}/{\beta^2}$. 
The following example is an application of Corollary \ref{coro panjer}.
\begin{proposition}
Let $X$ be a discrete variable on $\left\{\mesh i, i\in\N\right\}$, where $\delta\in ]0,1/\beta]$, with density
$$
p_i=\frac{\alpha(1-\mesh \beta)^{i-1}\Gamma\left(\frac{\alpha}{1-\mesh\beta}+i\right)}{\Gamma(i+1)\Gamma\left(\frac{\alpha}{1-\mesh\beta}+1\right)}p_0
$$
for $i\in\N\setminus \left\{0\right\}$ and $Z\sim \Gamma(\alpha,\beta)$. 
Then, we have
$
W_1(X,Z)\leq 2\delta.
$
\end{proposition}
\begin{proof}
The result follows from Corollary \ref{coro panjer}. The condition $\tau_i+\delta \overline{\mathrm{Id}}_i\geq 0$ is verified  because 
$$
\frac{\tau_i}{-\overline{\mathrm{Id}}_i}=\frac{\delta i}{\beta\delta i-\alpha}\geq \frac{1}{\beta}\geq \delta
$$
when $\overline{\mathrm{Id}}_i=\alpha/\beta -\delta i\leq 0$.
\end{proof}

Another example we can tackle easily is the comparison of gamma distribution with a negative binomial law where  $Y \sim \mathrm{NB}(n, p)$  if $p(i) = \binom{n+i-1}{n-1}p^n (1-p)^i$ for $i \in \mathbb N$. Using the same parameterization as in \cite{adell1994approximating} we obtain the following result. 

\begin{proposition}
    Let $Y$ be negative binomial with parameters $n\in \N_0$ and $p = \beta/(t+\beta)$ for $\beta, t >0$. 
    Let
    $$
    X=\frac{1}{\sqrt{t(\beta+t)}}\left(Y-\frac{tn}{\beta}\right)+\frac{n}{\beta}
    $$
    and $Z\sim \Gamma(n,\beta)$.
    Then 
    \begin{equation}
        \label{nbga}
W_1(X, Z)   \le  \frac{2}{\sqrt{t(\beta+t)}}.
    \end{equation}
\end{proposition}

\begin{proof}
The standardization of $X$ follows from \eqref{eq:autosta}.
The sequence of weights $(\pi_i)_{i \in \mathbb N}$ from \eqref{eq_ai3}  is given by 
    \begin{equation*}
        \pi_i = \frac{(n + i)(\sqrt{t(\beta+t)} -t)}{\beta i+n(\sqrt{t(\beta+t)} -t)}
    \end{equation*}
    for all $i \in \mathbb{N}$. 
    The sequence of weights is decreasing in $i$ and satisfies 
    $0 \le \pi_i \le 1$
         for all admissible $i, n, \beta$ and $t$. 
    A straightforward application of Theorem \ref{theo main intro} (Corollary \ref{cor:regsp}) yields the claim. 
\end{proof}

\begin{remark}
    The bound in \cite{adell1994approximating} is given in Kolmogorov distance and hence not directly comparable to ours; it may be interesting to note that \cite{adell1994approximating}'s bound does not depend on the parameter  $\beta$ so that the Kolmogorov distance does not tend to 0 as $\beta \to \infty$ whereas the Wasserstein distance does. 
\end{remark}

\subsection{Beta approximation}
\label{sn:beta}
Throughout this subsection, we let  $Z \sim \text{Beta}( \alpha, \beta)$ be a beta variable with parameters $\alpha,  \beta >0$ so that $q(x) \propto x^{\alpha-1} (1-x)^{\beta-1} \mathbb{I}[0<x<1]$. The beta distribution belongs to the  Integrated Pearson family,  with Stein kernel  $\tau_q(x)=x(1-x)/(\alpha+\beta)$.  As in the previous setting we could bound $\sup_{h\in \text{Lip}(1)}||(f'_h \tau_q)'||_{\infty} \leq 2$. To bound $C^0_q(\tau_q)$ and $C^1_q(\tau_q)$, we have the following result; the first statement comes from Lemma 3.4 in \cite{goldstein2013stein}.
\begin{lemma}
\label{lem beta bound}
Let $Z\sim \mathrm{Beta}(\alpha,\beta)$. 
We have
$$
\sup_{h\in \text{Lip}(1)}||f_h'||_{\infty}\leq (b_0+b_1)(\alpha+\beta),
$$
with $b_0$ and $b_1$ given in equations \eqref{eqn:b0} and \eqref{eqn:b1} in the Appendix, and if $\alpha,\beta<1$
\begin{equation}
\label{eq lem beta}
\sup_{h\in \text{Lip}(1)}||\tau_q'f_h'||_{\infty}\leq \frac{2}{1+\min(\alpha,\beta)}.
\end{equation}
\end{lemma}
We thus have
\begin{equation*}
    C^0_q(\tau_q) \leq \begin{cases}
    1+\frac{1}{1+\min(\alpha,\beta)}&\text{if  } \alpha,\beta<1\\
    1+\frac{b_0+b_1}{2}&\text{else}
    \end{cases},
    \quad C^1_q(\tau_q) \leq \frac{b_0+b_1}{3}
\end{equation*} 
and
$$
C_q(\tau_q) \leq 1+\frac{b_0+b_1}{2}.
$$
Now let $X \sim p$ be a discrete random variable. In order to satisfy \eqref{condition1} for $w = \tau_q$   it is necessary that  $\mathbb E[X] = \alpha/(\alpha+ \beta)$ and $\mathrm{Var}[X] = {\alpha \beta}/{((\alpha+\beta)^2(1+\alpha+\beta)}$. 


\begin{proposition}\label{prop:ourgoldrein}
Let $Y$ be a Polya variable with parameters $\alpha>0$, $\beta>0$, $m\geq 1$, $n\geq 1$ and $Z\sim \mathrm{Beta}(A,B)$ with $A=\alpha/m$ and $B=\beta/m$. 
Set
$$
X=\frac{1}{\sqrt{n\left(A+B+n\right)}} \left(Y-\frac{nA}{A+B}\right)+\frac{A}{A+B}.
$$
Then
\begin{equation}\label{boundw1wabe}
W_1(X,Z)\leq \left(1+\frac{b_0+b_1}{2}\right)\frac{1}{\sqrt{n\left(A+B+n\right)}} 
\end{equation}
and for $\alpha,\beta\leq 1$
\begin{equation*}
W_1(X,Z)\leq \frac{2+\min(\alpha,\beta)}{1+\min(\alpha,\beta)}\frac{1}{\sqrt{n\left(A+B+n\right)}} +\frac{b_0+b_1}{3}\frac{1}{n\left(A+B+n\right)}. 
\end{equation*}
\end{proposition}
\begin{proof}
The standardization of $X$ follows from \eqref{eq:autosta}.
Let $\mathcal{X}=\left\{x_0,\ldots,x_n\right\}$ be the set of values taken by $X$. We have
$$
x_0=\left(1-\frac{n}{\sqrt{n\left(n+A+B\right)}}\right)\frac{A}{A+B}>0
$$
and 
$$
x_n=\frac{n}{\sqrt{n\left(A+B+n\right)}} \left(1-\frac{A}{A+B}\right)+\frac{A}{A+B}<1.
$$
Hence, it holds that $\mathcal{X}\subset (0,1)$. 
 The coefficients $\pi_i$ from \eqref{eq_ai3} become
$$
\pi_i=\frac{(i + A) (-i + n) \left(B (B + n - \sqrt{n (A + B + n)}) + 
   A (B - n + \sqrt{n (A + B + n)})\right)}{\left(A (-i + n) + 
   B (-i + \sqrt{n (A + B + n)})\right) \left( Bi + 
   A (i - n + \sqrt{n (A + B + n)})\right)}.
$$
In the Appendix we show that  \begin{equation}\label{eq:claim01}
    0\leq \pi_i\leq 1, \quad i=0,\ldots,n
\end{equation}
so that $\E[|\pi(X)|+|1-\pi(X)|]=1$.
(We actually give two proofs of this fact, the first being based on direct computations and the other  based on Propositions \ref{prop borne inf} and \ref{prop borne sup} that doesn't require to compute the coefficient $\pi_i$.) 
A straightforward application of Theorem~\ref{theo2dform} and Corollary~\ref{cor:regsp} yields the claims. 
\end{proof}

\begin{remark}\label{rem:polyaeg}
As mentioned in the introduction, \cite{goldstein2013stein} have already studied the same problem (through a different version of Stein's method). They obtain (see their Theorem 1) 
\begin{equation*}
W_1\left(\frac{Y}{n},Z\right)\leq \frac{1}{nm}\left(\frac{m+\max(\alpha,\beta)}{2}+\frac{\alpha\beta}{\alpha+\beta}\right)(b_0+b_1)+\frac{3}{2n} =: \mathrm{GR}(n).
\end{equation*}
Since $\E[Y]=nA/(A+B)$, we have  \eqref{rema1}
\begin{equation*}
W_1\left(\frac{Y}{n},X\right) 
=\left(1-\frac{n}{\sqrt{n(n+A+B)}}\right)\frac{2A}{A+B}. 
\end{equation*}
Combining this last inequality with  \eqref{boundw1wabe} and the triangle inequality  we conclude 
\begin{equation*}
W_1\left(\frac{Y}{n},Z\right)\leq \frac{b_0+b_1+2}{2\sqrt{n\left(\frac{\alpha}{m}+\frac{\beta}{m}+n\right)}} +\left(1-\frac{n}{\sqrt{n(n+\frac{\alpha}{m}+\frac{\beta}{m})}}\right)\frac{2\alpha}{\alpha+\beta} =: \mathrm{GS}(n).
\end{equation*}
For the sake of illustration, fix $\alpha\ge \beta$. Direct computations lead to 
\begin{align*}
\lim_{n \to \infty} n \mathrm{GR}(n)& =  \frac32   + \frac{b_0+b_1}{2} +  \frac{\alpha}{m}\left(\frac{b_0+b_1}{2}\,\frac{ \alpha+3 \beta}{\alpha + \beta}\right) \\
    \lim_{n \to \infty} n \mathrm{GS}(n)& =1  +  \frac{b_0+b_1}{2}+ \frac{\alpha}{m}
\end{align*}
hence both bounds are of the same order. Since $b_1 \ge 2$ for all values of $\alpha$ and $\beta$, the latter is always less than the former. 
\end{remark}

\begin{proposition}\label{semici}
Let $Y$ be a random variable defined on $\left\{1,\ldots, n-1\right\}$ with density
\begin{equation*}
p_i=\frac{1}{\begin{pmatrix}n\\ 2\end{pmatrix}}\frac{\prod_{j=1}^{i-1}2j+1}{\prod_{j=1}^{i-1}2j}\frac{\prod_{j=1}^{n-i-1}2j+1}{\prod_{j=1}^{n-i-1}2j}.
\end{equation*}
and $Z\sim\mathrm{Beta}(3/2,3/2)$. 
Set
$$
X=\frac{1}{\sqrt{n^2-n-2}}\left(Y-\frac{n}{2}\right)+\frac{1}{2}.
$$
Then
\begin{equation*}
W_1(X,Z)\leq 5\frac{1}{\sqrt{n^2-n-2}}. 
\end{equation*}

\end{proposition}
\begin{proof}
The standardization of $X$ follows from \eqref{eq:autosta}. Let $\mathcal{X}=\left\{x_1,\ldots,x_{n-1}\right\}$ be the set of values taken by $X$. Simple computations show that
$x_0>0$ and $x_{n-1}<1$ if and only if $n\geq 2$. Hence, it holds that $\mathcal{X}\subset ]0,1[$. 
The coefficients $\pi_i$ from \eqref{eq_ai3} become
$$
\pi_i=-\frac{(2 i + 1) (n - i - 1)}{4 i^2 - 4 in + n + 2}.
$$
 We show in the appendix  that $\pi_i\in [0,1]$ for all $i=1,\ldots,n-1$.  
Since  $b_0+b_1=8$, we get the result by Theorem \ref{theo2dform}. 
\end{proof}

\begin{remark}
 Let $Z\sim\mathrm{Beta}(3/2,3/2)$. In \cite{fulman2014stein}, it is shown that
$W_1\left({Y}/{n},Z\right)\leq {59}{/(2n)}.$
Our bound, combined with  \eqref{rema1}, gives 
\begin{equation*}
W_1\left(\frac{Y}{n},Z\right)\leq 5\frac{1}{\sqrt{n^2-n-2}} +\left(\frac{n}{\sqrt{n^2-n-2}}-1\right). 
\end{equation*}
which  is of the same order but with a constant equal to  11/2.
\end{remark}

Now, we look at another example described in \cite{fulman2023beta}, concerning a Beta approximation for the stationary distribution in the two alleles Moran model.
\begin{proposition}
\label{prop:moran}
Let $Y$ be a discrete variable defined on $\left\{0,1,\ldots,2n-1, 2n\right\}$ with mass function given by
\begin{equation*}
 p_i=p_0\frac{(2 n)! \Gamma(i + A)\Gamma(B - i)}{i! (2 n - i)! \Gamma(A) \Gamma(B)},
\end{equation*}
where  $A=2na/(2n-a-b)$, $B=2n(2n-a)/(2n-a-b)$ for some $a, b>0$ and $p_0$ ensures the normalisation, and let $Z\sim\mathrm{Beta}(a,b)$. 
Set
\begin{equation*}
    \delta_M = 
     \sqrt{\frac{1}{4 n^2} - \frac{a + b}{8 n^3 (a + b + 1)}}
\end{equation*}
and consider 
\begin{equation*}
X=\delta_M\left(Y - \frac{2 n a}{a + b}\right)  + 
 \frac{a}{a + b}.
\end{equation*}
For every $a,b>1$, there exists $n_0\in\N$ such that for all $n\geq n_0$ 
\begin{equation}
\label{eqmoo}W_1(X,Z)\leq \left(1+\frac{b_0+b_1}{2}\right) \delta_M  
\end{equation}
\end{proposition}
\begin{remark}
    We stress that the restriction $a, b>1$  in \eqref{eqmoo} is purely an artefact of our proof: we need to use the fact that  weights  $\pi_i$ (which are not explicit in this case) are between 0 and 1, and though we observe this numerically for all parameter values, we were only able to prove it  when $a, b>1$.   Hence we conjecture that \eqref{eqmoo} holds for all parameter values. 
\end{remark}
\begin{remark}
    In \cite{fulman2023beta} the version of   Stein's method  developed in \cite{dobler2015stein} is applied to the same problem and yields  $d_2(Y/(2n), Z) \le C(a, b)/n$  for all positive $a, b$,  where $d_2(\cdot, \cdot)$ is a smooth Wasserstein metric (i.e.\ the supremum in \eqref{eqwa} is taken over Lipschitz functions with bounded second derivative) and $C(a, b)$ is an explicit constant. Combining  \eqref{rema1} with \eqref{eqmoo} we obtain 
    \begin{equation*}
         W_1\left( \frac{Y}{2n}, Z\right) \le \frac{2a}{a+b} \left(1- 2 n \delta_M\right) + \left( 1 + \frac{b_0 + b_1}{2}\right) \delta_M.
    \end{equation*}
    This bound is of the form  $\widetilde{C}(a, b)/n$ with $\widetilde{C}(a, b)$ an explicit constant; obviously this  also provides a  bound on  the smooth Wasserstein distance $d_2(Y/(2n), Z)$. We note  however that   the  constant $C(a, b)$ from \cite{fulman2023beta} is not competitive with respect to ours (e.g.\ when $a = 2$ and $b=3$ then 
    $C(a, b) = 650$ whereas $\widetilde{C}(a, b) = 9$).  
\end{remark}

\subsection{Normal approximation} \label{sn:normalapprox}
Throughout this subsection, we let $Z \sim \varphi$ with $\varphi$  the standard Gaussian density. The normal distribution belongs to the Integrated Pearson family, with Stein kernel  $\tau_\varphi(x)  = 1$.  Clearly  $$C_\varphi^0(1) = 1 \mbox{ and } C_\varphi^1(1) = 0.$$
Now let $X \sim p$ be a discrete random variable.  In order to satisfy \eqref{condition1} for $w = \tau_q$   it is necessary that  $\mathbb E[X] = 0$ and $\mathrm{Var}[X] = 1$. 

We begin by studying the behavior of the sequence of weights $\pi$ for sums of independent variables. Perhaps unsurprisingly, we obtain  a  result  reminiscent of equation (14) in \cite{courtade2019existence} or \cite{nourdin2014integration} (dealing  with Stein kernels of sums of independent variables), as follows. 
\begin{lemma}
\label{lem a}
Let  $Y_{j}, j = 1, \ldots, n$ be  independent integer valued random variables. Let $\E [Y_{j}] = \mu_{j}$ and $\mathrm{Var}[Y_{j}] = \sigma^2_{j}$ and  set $X_{j} = (Y_{j} - \mu_{j})/\sigma_{j}$ for $j = 1, \ldots, n$.  Define  $X= \sum_{j=1}^n (Y_{j}- \mu_{j}) / s$ with $s^2 = \sum_{j=1}^n \sigma_{j}^2$.   
For each $j = 1, \ldots, n$ let  $\pi_{j}$ be the weights defined by \eqref{eq_ai3gau} for $X_{j}$. 
Then 
\begin{equation}\label{eq:pisum}
\pi(X):=\sum_{j=1}^n\frac{\sigma_{j}^2}{s^2}\E\left[\pi_{j}(X_{j})|X\right]
\end{equation}
satisfies \eqref{eq stein a} for $X$ and $Z$. 
\end{lemma}
\begin{proof}
 For the sake of clarity, in this proof we will use the notation $\Delta_{\delta}^+ f(x)=(f(x+\delta)-f(x))/\delta$ and similarly for $\Delta_{\delta}^-$ and $\Delta^{\pi}_{\delta}$. By construction,  \eqref{eq stein a} for each $X_{j}$ reads 
$\mathbb{E}[\Delta^{\pi_{j}}_{1/\sigma_{j}} g(X_{j})] = \E [X_{j}g(X_{j})]$
 for all bounded function $g$. Set $\hat{X}_{i}=\sum_{j=1,j\neq i}^{n} \frac{\sigma_{j}}{s}X_{j}$. Then for $f$ bounded we have
\begin{align*}
\E\left[Xf(X) \right]
&=\sum_{j=1}^n\frac{\sigma_{j}}{s}\E\left[ X_{j} f(X)\right]\\
&=\sum_{j=1}^n\frac{\sigma_{j}}{s}\E\left[\E\left[X_{j} f\left(\frac{\sigma_{j}}{s}X_{j}+\hat{X}_{j}\right) |\hat{X}_{j}\right]\right].
\end{align*}
Fix $j \in \left\{1, \ldots, n\right\}$.  By independence we can use equation \eqref{eq stein a} applied to $X_{j}$ with the function $g_j(x) = f(\sigma_{j}/s x + \hat X_{j})$ to get 
\begin{align*}
\E\left[Xf(X)\right]&
= \sum_{j=1}^n\frac{\sigma_{j}}{s} \E\left[\E \left[   \Delta^{\pi_{j}}_{1/\sigma_{j}} g_j(X_{j})| \hat{X}_{j}\right]  \right].
\end{align*}
Now, remark that 
\begin{align*}
\Delta^{\pi_{j}}_{1/\sigma_{j}} g_j(X_{j})  & = \pi_{j}(i)  \frac{f \left(\frac{\sigma_{j}}{s}(X_{j} + \frac{1}{\sigma_{j}})+\hat{X}_{j}\right)-f \left(\frac{\sigma_{j}}{s}X_{j}+\hat{X}_{j}\right) }{1/\sigma_{j}} \\
& \quad + (1-\pi_{j}(X_{j})) \frac{f \left(\frac{\sigma_{j}}{s} X_{j} +\hat{X}_{j}  \right)-f \left(\frac{\sigma_{j}}{s}(X_{j}-\frac{1}{\sigma_{j}}) +\hat{X}_{j}\right)}{1/\sigma_{j}}\\
&= \frac{\sigma_{j}}{s}\pi_{j}(X_{j})  \frac{f \left(X+\frac{1}{s}\right)-f \left(X\right) }{1/s}\\
&\quad + \frac{\sigma_{j}}{s}(1-\pi_{j}(X_{j})) \frac{f \left(X \right)-f \left(X-\frac{1}{s} \right)}{1/s}\\
&=\frac{\sigma_{j}}{s}\left[\pi_{j}(X_{j})\Delta^{+}_{{1/s}} f\left(X\right) 
+\left(1-\pi_{j}(X_{j})\right)\Delta^{-}_{{1/s}} f\left(X\right)\right].
\end{align*}
Hence, we have
\begin{align*}
\E\left[Xf(X)\right] &=\sum_{j=1}^n\frac{\sigma_{j}^2}{s^2} \E\left[\E \left[\pi_{j}(X_{j})\Delta^{+}_{{1/s}} f\left(X\right) 
+\left(1-\pi_{j}(X_{j})\right)\Delta^{-}_{{1/s}} f\left(X\right)| \hat{X}_{j}\right]  \right]\\
&=\sum_{j=1}^n\frac{\sigma_{j}^2}{s^2} \E\left[\pi_{j}(X_{j})\Delta^{+}_{{1/s}} f\left(X\right) 
+\left(1-\pi_{j}(X_{j})\right)\Delta^{-}_{{1/s}} f\left(X\right)\right].
\end{align*}
By conditioning on $X$ and noting that $s^2=\sum_{j=1}^n \sigma_{j}^2$, we obtain
\begin{align*}
\E\left[X f(X)\right]
&= \E\left[\sum_{j=1}^n\frac{\sigma_{j}^2}{s^2}\E[\pi_{j}(X_{j})|X]\Delta^{+}_{{1/s}} f\left(X\right)
+\left(1-\sum_{j=1}^n\frac{\sigma_{j}^2}{s^2}\E[\pi_{j}(X_{j})|X]\right)\Delta^{-}_{{1/s}} f\left(X\right)\right]
\end{align*}
which gives the claim. 
\end{proof}

A decomposition  such as \eqref{eq:pisum} is of course very handy for producing applications of Theorem \ref{theo main intro} to normal approximations of sums of independent discrete random variables. In particular, the following holds. 

\begin{theorem}
\label{theo TCL}
Under the  notations from Lemma \ref{lem a}, we have 
\begin{equation}
W_1(X,Z)\leq \frac{1}{s^3}\sum_{j=1}^n \sigma_{j}^2\E\left[|\pi_{j}(X_{j})|+\left|1-\pi_{j}(X_{j})\right|\right].
\end{equation}
\end{theorem}
\begin{proof}
Using Lemma \ref{lem a}, we obtain
\begin{align*}
\E[|\pi(\s)|+|1-\pi(\s)|]
&=\E\left[\left|\sum_{j=1}^n\frac{\sigma_{j}^2}{s^2}\E\left[\pi_{j}(X_{j})|\s\right]\right|+\left|\sum_{j=1}^n\frac{\sigma_{j}^2}{s^2}\E\left[1-\pi_{j}(X_{j})|\s\right]\right|\right]\\
&\leq\sum_{j=1}^n\frac{\sigma_{j}^2}{s^2}\E\left[\E\left[|\pi_{j}(X_{j})||\s\right]+\E\left[\left|1-\pi_{j}(X_{j})\right||\s\right]\right]\\
&=\sum_{j=1}^n\frac{\sigma_{j}^2}{s^2}\E\left[|\pi_{j}(X_{j})|+\left|1-\pi_{j}(X_{j})\right|\right]. 
\end{align*}
If $\E[|\pi(X)|]=\infty$, the statement is trivial. Otherwise, we have that $\pi$, which solves equation \eqref{eq stein a} by Lemma \ref{lem a}, is also solution of equation \eqref{eq_deltas}. A direct application of Theorem \ref{theo main intro} yields the claim.
\end{proof}

\begin{corollary}
\label{coro TCL}
If, moreover,  the $(Y_{j})_{j =1, \ldots, n}$ are identically distributed with mean $\mu$ and variance $\sigma^2$ then 
\begin{equation}
\label{eq coro TCL}
W_1(X,Z)\leq \frac{1}{\sigma\sqrt{n}}\E\left[|\pi_{1}(X_{1})|+\left|1-\pi_{1}(X_{1})\right|\right].
\end{equation}
\end{corollary}

If $0\leq \pi_{1}(X_{1})\leq 1$  then $W_1(X,Z)\leq 1/(\sigma \sqrt{n})$. Hence also in the case of normal approximation of the binomial; this  is of course a very classical application, see e.g.\ \cite{feller1945normal} for an overview already back in 1945.
We can directly compare our Theorem~\ref{theo TCL} with specific instantiations of Proposition 2.2 in \cite{goldstein20071} which state that 
if  the summands  are identically distributed then 
\begin{equation}
\label{eq Goldstein}
W_1(X,Z)\leq c \frac{\E[|Y_{1}-\mu|^3]}{\sigma^3\sqrt{n}}
\end{equation}
where $c\leq 3$.
\begin{example} Clearly the bounds on $W_1(X,Z)$ given  by Corollary \ref{coro TCL} and equation \eqref{eq Goldstein} are of the same order; all that remains is to compare the constants.  We use the weights provided in Example \ref{ex:xomepisnorm}. We denote $\Delta_{\mathrm{GS}}$ the  constant appearing in  \eqref{eq coro TCL}, and $\Delta_{\mathrm G}$ that appearing in  \eqref{eq Goldstein}. The following then holds. 
\begin{enumerate}
    \item If {$Y_{1}, \ldots, Y_{n}$} are iid  Bernoulli variables with parameter $t\in(0, 1)$, 
one  gets
\linebreak
$
\Delta_{\mathrm{GS}} = {1}/{\sqrt{t(1-t)}}$ (this follows immediately because  $\pi_1 \in (0,1)$) and $\Delta_{\mathrm G} = ({t^2+(1-t)^2})/{\sqrt{t(1-t)}}$, 
 so our bound is never an improvement. 
\item If $Y_{1}, \ldots, Y_{n}$ are iid  Poisson variable with parameter $\lambda>0$, 
then we get 
$
\Delta_{\mathrm{GS}} = {1}/{\sqrt{\lambda}}{=\E[(X-\lambda)^3]/\lambda^{3/2}}$ (again,  $\pi_1 \in (0,1)$) and $\Delta_{\mathrm G} = c {\E[|X_1-\lambda|^3]}/{\lambda^{3/2}}
$
; taking $c=3$ the latter, for which we have not found a simple expression,  is larger than the former for  {all} values of $\lambda$. 
\item If $Y_{1}, \ldots, Y_{n}$ are iid uniform variable on $\left\{0,\ldots,2\ell\right\}$ for $\ell\in\N$,
 then   $\Delta_{\mathrm{GS}} = \sqrt3 (4+\ell+\ell^2)/(\sqrt{\ell+\ell^2}(2+4\ell))$ whereas 
$  \Delta_{\mathrm G} =  c 3 \sqrt 3{\sqrt{\ell+\ell^2)}}/{(2+4\ell)}$.
 For large $\ell$, it holds that $\Delta_{\mathrm{GS}}/\Delta_{\mathrm G} \approx 1/(3c)$. 

\item If $Y_{1}, \ldots, Y_{n}$ are iid Negative Binomial variables with parameters $\ell$ and $t$,  $\Delta_{\mathrm{GS}} = (2-t)/(\sqrt{\ell(1-t)})$. The constant $  \Delta_{\mathrm G}$ is difficult to compute in this case.
\end{enumerate}
\end{example}

We now tackle examples which do not follow directly from a CLT. 
\begin{example}\label{exa:hyperggg}
Let $Y$ be a hypergeometric  variable with parameters $r, n, N$, $r\geq n$ and $N > n+r$. Set $X = (Y-\mu)/\sigma$ where $\mu$ and $\sigma$ are defined as in \eqref{hypergeomai}.  
Using the $\pi_i$ given in \eqref{hypergeomai} we get  from Corollary \ref{cor:regsp}  
$$
W_1(X,Z)\leq \sqrt{\frac{N^2(N-1)}{nr(N-n)(N-r)}}.
$$
Returning in the original standardization we get, from \eqref{rema1} 
$$
\frac{nr}{N} \le W_1(Y,Z)\leq \frac{nr}{N} + \sqrt{\frac{N^2(N-1)}{nr(N-n)(N-r)}}.
$$
which is of precisely the correct order. This problem is classical,  see e.g.\ \cite{nicholson1956normal}, though we have not found competitor to ours in the literature.

\end{example}

Now, we look at an application to non integer valued random variables, with a non regularly spaced support; this example is inspired from \cite{mckeague2019stein,mckeague2021stein}, and \cite{chen2023optimal}.
\begin{proposition}
Take $n\in\N$ and let $\mathcal{X}=\left\{x_i,\, i=1,\ldots,n\right\}$ be the unique set of size $n$ such that $x_i< x_{i+1}$, 
$$
x_{i+1}-x_i= -\frac{1}{\sum_{j=1}^i x_j}, \quad i=1, \ldots, n-1
$$
and $\sum_{i=1}^n x_i=0$.
Let $X$ be a random variable uniformly distributed on $\mathcal X$. Then  
$$
W_1\left(\sqrt{\frac{n}{n-1}}X,Z\right)
\leq 4\left(\frac{n}{n-1}\right)^{3/2}\frac{\sqrt{\ln(n)}}{n}
$$
for all $n >100$.
\end{proposition}
\begin{proof}
By Lemma 2.1 in \cite{chen2023optimal}, we know that $\Var[X]=\sum_{i=1}^nx_i^2/n={(n-1)}/{n}$. Hence, to satisfy  condition \eqref{condition1-intro} we have to consider the variable $\epsilon X$ with $\epsilon^2=\frac{n}{n-1}$. Remark that
$$
\pi_i=\epsilon^2\sum_{j=1}^i (x_i-x_j)x_j+i
$$
and
\begin{align*}
\pi_{i+1}-\pi_i=\epsilon^2(x_{i+1}-x_i)\sum_{j=1}^i x_j+1=1-\epsilon^2=-\frac{1}{n-1}.
\end{align*}
Hence, we have $\pi_i=1-(i-1)/(n-1)$. 
Now, we compute
\begin{align*}
W_1(\epsilon X,Z)&\leq \frac{\epsilon}{n}\sum_{i=1}^n\left\{|\pi_i|(x_{i+1}-x_i)+|1-\pi_i|(x_i-x_{i-1})\right\}\\
&= -\frac{\epsilon}{n}x_1(|\pi_1|+|1-\pi_2|)+\frac{\epsilon}{n}\sum_{i=2}^{n-1} x_i\left\{|\pi_{i-1}|+|1-\pi_i|-|\pi_i|-|1-\pi_{i+1}|\right\}\\
&\quad+\frac{\epsilon}{n}x_n(|\pi_{n-1}|+|1-\pi_n|)\\
&= -\frac{\epsilon}{n}x_1(1-(\pi_2-\pi_1))+\frac{\epsilon}{n}\sum_{i=2}^{n-1} x_i\left\{(\pi_{i+1}-\pi_i)-(\pi_i-\pi_{i-1})|\right\}\\
&\quad+\frac{\epsilon}{n}x_n(1-(\pi_n-\pi_{n-1}))\\
&=\frac{\epsilon^3}{n}(x_n-x_1)=\frac{2\epsilon^3}{n}x_n.
\end{align*}
By Lemma 3.2 in \cite{chen2023optimal}, we have $x_n\leq \sqrt{2(1+\ln(m))}$ where $m=n/2$ if $n$ is even and $m=(n+1)/2$ if $n$ is odd, when $n>100$. This entails that $x_n \leq 2\ln(n)$ for $n>100$.
\end{proof}

We conclude the section in anticlimactic fashion, with an interesting example which we so far were not able to tackle. 
\begin{remark}\label{rem:vol}
 Let $V_i(K)$ denote the $i^{\text{th}}$ intrinsic volume of  a convex body $K\subset \R^n$ (see e.g.\ \cite{garino2023total} for a definition and some context). If $(K_n)_{n\in\N}$ is a sequence of convex bodies of dimension $n$ with constant diameter and if $Y_n$ is a random variable defined on $\left\{0,\ldots,n\right\}$ with mass function $p_{i}\propto V_i(n^{\alpha }K_n)$, the question of whether $(Y_n-E[Y_n])/\sqrt{\Var[Y_n]}$ converges to a standard Gaussian variable $Z$ when $n\to\infty$ remains open. In the case of the unit ball $B_n\subset \R^n$, we have
$V_i(n^{\alpha}B_n)=n^{\alpha i}\binom{n}{i}\pi^{i/2}{\Gamma\left(1+\frac{n-i}{2}\right)}/{\Gamma\left(1+\frac{n}{2}\right)}$
for $i=0,\ldots,n$. Numerically, we see that in this case,  when $\alpha< 1/2$ and $n$ is large enough,  then the coefficients from \eqref{eq_ai3gau}   satisfy $0\leq \pi_i \leq 1$ for all $i=0,\ldots,n$. If we could verify this it would follow  that $W_1((Y_n-E[Y_n])/\sqrt{\Var[Y_n]},Z)\leq 1/\sqrt{\Var[Y_n]}$ under these conditions. However, the coefficients $(\pi_i)$ cannot be computed analytically and the Propositions \ref{prop borne inf} and \ref{prop borne sup}, although the  conditions seem verified, are difficult to use since $p_{i+1}/p_{i}$  doesn't have a simple expression. Observe that the necessary conditions \eqref{necessary gaus} are satisfied since in this case it holds from a convex domination argument that   $\Var[Y_n]\leq \E[Y_n]$   (Arturo Jaramillo, private communication) and   
$  \E[Y_n]+\Var[Y_n]\leq 2V_1(\alpha^n B_n)=O(n^{\alpha+1/2})
$ (see Corollary 2.8 in \cite{garino2023total}).  
\end{remark}

\subsection{Stationary distribution of Erlang-C system}\label{sn:queue}

Consider the stationary distribution of the Erlang-C queuing system
\begin{equation*}
p_i=\begin{cases}
p_0\left(\frac{\lambda}{\mu}\right)^{i}\frac{1}{i!}&\text{  if  }i=0,\ldots,n\\
p_0\left(\frac{\lambda}{\mu}\right)^{i}\frac{1}{n^{i-n}n!}&\text{  if  }i=n+1,\ldots
\end{cases}
\end{equation*}
where $n\in \N$, $\lambda,\mu>0$ with $\lambda<\mu n$. Let $X\sim p$ and $Z$ be a continuous variable with density given by 
$$q(x)=C\exp\left(\frac{1}{\mu}\int_0^x b\right), \quad x\in \R$$ 
with $C>0$ a normalisation constant and $$b(x)=\mu\max\left\{-x,\sqrt{\frac{\mu}{\lambda}}(\frac{\lambda}{\mu}-n)\right\}.$$ It is known that (see Theorem 1 in \cite{braverman2017stein})
$$
W_1\left(\sqrt{\frac{\mu}{\lambda}}\left(X-\frac{\lambda}{\mu}\right),Z\right)\leq 205\sqrt{\frac{\mu}{\lambda}}.
$$ 
Direct  application of our Theorem \ref{theo main intro} does not yield a competitive bound however a more refined approach  (anticipated in Theorem \ref{rem:notroutine}) yields the following result. 
\begin{proposition}
Let $X$ and $Z$ be defined as above. We have
$$
W_1\left(\sqrt{\frac{\mu}{\lambda}}\left(X-\frac{\lambda}{\mu}\right),Z\right)\leq  31\sqrt{\frac{\mu}{\lambda}}. 
$$
\end{proposition}
\begin{proof}
As a Stein operator for $Z$, we can use 
$
\T_Z f = \mu f'+bf
$ (see equation (3.2) in \cite{braverman2017stein}). So, in this example, we have $w =\mu$, $s=b$ and $x_i=\sqrt{{\mu}/{\lambda}}\left(i-{\lambda}/{\mu}\right)$ for $i \in \mathbb N$. We can compute that $\E[b(X)]=0$, $\E[Xb(X)]+\mu=0$ and $\pi_i=1$ for all $i\in\N$ (i.e.\ here the bespoke derivative is none other than the right derivative). We have $\lim_{i\to\infty}\pi_{i+1}p_{i+1}w_{i+1}/(\pi_ip_iw_i)=\lambda/(\mu n)<1$ and thus $\pi w\in L^1(p)$. 
We also know that (see Lemma 3 in \cite{braverman2017stein}) 
$$
|f''_h(x)|\leq 
\begin{cases}
(23+\frac{13}{x_n})/\mu&\text{  if  }x\leq x_n\\
2/\mu&\text{  if  }x\geq x_n.
\end{cases}
$$ 
Hence, returning to the notations introduced in Theorem \ref{rem:notroutine}, we have $C^{0,i}_q(w)\leq (23+{13}/{x_n})$ when $i=0,\ldots,n$ and $C^{0,i}_q(w)\leq 2$ when $i\geq n+1$. Using Theorem \ref{rem:notroutine}, 
we get
\begin{align*}
W_1\left(\sqrt{\frac{\mu}{\lambda}}\left(X-\frac{\lambda}{\mu}\right),Z\right)&\leq \frac{1}{2}\sqrt{\frac{\mu}{\lambda}}\sum_{i=0}^\infty\left\{C^{0,i+1}_q(w)|\pi_i|+C^{0,i}_q(w)|1-\pi_i|\right\}p_i\\
&\leq \frac{1}{2}\sqrt{\frac{\mu}{\lambda}}\left(\left(23+\frac{13}{x_n}\right)\sum_{i=0}^{n-1} p_i
+2\sum_{i=n}^{\infty}p_i\right)\\
&\leq \frac{1}{2}\sqrt{\frac{\mu}{\lambda}}\left(23+13\left(2+\sqrt{\frac{\mu}{\lambda}}\right)\right)
\end{align*}
where we have used the fact that $\mathbb{P}[X\leq x_n]\leq (2+\sqrt{\mu/\lambda})x_n$ (see equation (3.15) in \cite{braverman2017stein}). As we assume $\mu\leq\lambda$, we have the desired conclusion.
\end{proof}

\section*{Acknowledgements}
We thank Ivan Nourdin for drawing our attention to the  interesting question in Remark \ref{rem:vol}.  Also, we thank Anum Fatima and Gesine Reinert for drawing our attention to their paper  \cite{fatima2022stein} which  contains a generalisation of the bound from \cite{goldstein2013stein} allowing, like us,  to compare discrete and  continuous distributions on the real line, in so-called bounded Wasserstein. We have so far not been able to tackle their example (see their Section 3.2.2) with our methods. Finally, we note that a lot of computations were  performed through the software \texttt{Mathematica}; interested readers are invited to browse the webpage https://yvik.swan.web.ulb.be/research.html where they can download the corresponding notebook. 

\bibliographystyle{plain}
\bibliography{weibibl}

\appendix

\section{Proofs}
\label{sec:proofs}

\begin{proof}[Proof of Proposition \ref{prop:pstinope}]
If $I$ is finite, the assertion is trivial, so we consider $I=\N$. Let $f\in \mathcal{X}^{\star}$ be bounded. We have to show that $\E\left[w(X)\Delta^{\pi} f(X)\right]=\E\left[\frac{(\Delta^{\pi})^* (wp)}{p}(X)f(X)\right]$. By the Fubini theorem for infinite series, this is true if $\sum_{i,j}\left|f(x_j)\Delta^{\pi}_{ij}w(x_i)p_i\right|<\infty$. In fact, we have
\begin{align*}
\sum_{i=0}^{\infty}\sum_{j=0}^{\infty}\left|f(x_j)\Delta^{\pi}_{ij}w(x_i)p_i\right|&=\sum_{i=0}^{\infty}|\pi_if(x_{i+1})|w(x_i)p_i+\sum_{i=0}^{\infty}(1-2\pi_i)f(x_{i})|w(x_i)p_i\\
&\quad +\sum_{i=1}^{\infty}|(\pi_i-1)f(x_{i-1})|w(x_{i})p_i\\
&\leq ||f||_{\infty}  \left(4\E[|\pi(X)w(X)|]+2\E[|w(X)|]\right). \qedhere
\end{align*}
\end{proof}

\begin{proof}[Proof of Proposition \ref{prop:qsteinope}]
See Proposition 2.14 in \cite{dobler2015stein} and observe that the expression of $S$ in Remark 2.15 in \cite{dobler2015stein} is simply $\tau_q/w$ in our notations. 
\end{proof}

\begin{proof}[Proof of Proposition \ref{prop a}]
First, we consider $I=\left\{0,\ldots,\ell\right\}$. 
As \eqref{eq_deltas} is equivalent to \eqref{eq stein a}, choosing $f(x)=1$ and $f(x)=x$ in \eqref{eq stein a}, we obtain \eqref{condition1} because $\Delta^{\pi} 1=0$ and $\Delta^{\pi} x=1$. 
Hence, the equations \eqref{condition1} are necessary and we have to show that they are sufficient. Set $\delta_i:=\delta_i^-=x_i-x_{i-1}$. Equation \eqref{eq_deltas} means
\begin{equation}
\label{truc1}
-\frac{\pi_0}{\delta_1}w_0p_0+\frac{\pi_1-1}{\delta_1}w_1p_1=-s_0p_0,
\end{equation}
\begin{equation}
\label{truc3}
\frac{\pi_{\ell-1}}{\delta_{\ell}}w_{\ell-1}p_{\ell-1}+\frac{1-\pi_{\ell}}{\delta_{\ell}}w_{\ell}p_{\ell}=-s_{\ell}p_{\ell},
\end{equation}
and
\begin{equation}
\label{truc2}
\frac{\pi_{i-1}}{\delta_i}w_{i-1}p_{i-1}+\left(\frac{1}{\delta_i}-\pi_i\left[\frac{1}{\delta_i}+\frac{1}{\delta_{i+1}}\right]\right)w_ip_i+\frac{\pi_{i+1}-1}{\delta_{i+1}}w_{i+1}p_{i+1}=-s_ip_i
\end{equation}
for all $i=1,\ldots, \ell-1$. By computation, we can see that
the coefficients $(\pi_i)$ given in \eqref{eq_ai3} satisfy \eqref{truc1} and \eqref{truc2} for all $i=1,\ldots, \ell-1$. Since, we want $\pi_{\ell}=0$, we must have
$$
\sum_{j=0}^{\ell} \left((x_j-x_\ell)s(x_j)+w(x_j)\right)p_j=0.
$$
Moreover, equation \eqref{truc3} is equivalent to
$$
\sum_{j=0}^\ell \left((x_j-x_{\ell-1})s(x_j)+w(x_j)\right)p_j=0.
$$
The last two equations are equivalent to \eqref{condition1}. Hence, if they are satisfied, the equations \eqref{truc1}, \eqref{truc3} and \eqref{truc2} are verified and we have $\pi_0=1$ and $\pi_{\ell}=0$.

Now, we assume that $I=\N$. Equation \eqref{eq_deltas} means that equation \eqref{truc1} is satisfied as well as equation \eqref{truc2} for all $i\geq 1$. By computation, we can see that
the coefficients $(\pi_i)$ given in \eqref{eq_ai3} satisfy these equations. 

If $\pi w\in L^1(p)$, it entails that equation \eqref{eq stein a} holds for all bounded $f\in\R^{\ell}$. In particular, if we take the functions $f=1$ and $f_n(x)=x1_{[-n,n]}(x)$, we obtain $\E[s(X)]=0$ and $\E[w(X)\Delta^{\pi}f_n(X)]=-\E[f_n(X)s(X)]$. As $|\Delta^{\pi}f_n|\leq 1$ and $w\in L^1(p)$, we have $\E[w(X)\Delta^{\pi}f_n(X)]\to \E[w(X)]<\infty$ by dominated convergence. Since $-s$ is increasing, $(-f_ns)_{n\in\N}$ is an increasing sequence of functions (for $n$ large enough) integrable with respect to $X$ and which converges simply to $-\mathrm{Id}\, s$. 
Hence, we have  $-\E[f_n(X)s(X)]\to -\E[Xs(X)]<\infty$ by monotone convergence and $\E[w(X)]=-\E[X s(X)]$. \qedhere

\end{proof}

\begin{proof}[Proof of Proposition \ref{w=tau2}]
By equation \ref{eq:w=tau}, we know that $\E[X]=\E[Z]$ and $\Var[X]=\E[\tau_q(X)]$ which can be written as $\E[X^2]-\E[X]^2= \alpha \E[X^2]+\beta \E[X]+\gamma$. Recall that the Stein kernel satisfies $\Var[Z]=\E[\tau_q(Z)]$. Hence, we have
$$
(1-\alpha)\E[X^2]=\E[X]^2+\beta\E[X]+\gamma=\E[Z]^2+\beta\E[Z]+\gamma=(1-\alpha)\E[Z^2].
$$ 
We deduce that $\Var[X]=\E[X^2]-\E[X]^2=\E[Z^2]-\E[Z]^2=\Var[Z]$.
\end{proof}

\begin{proof}[Proof of Proposition \ref{prop:gaussthirdmom}]
We compute
\begin{align*}
\E\left[\pi\left(Y\right)\right]&=\sum_{i=0}^{\ell} \sum_{j=0}^i \left(\frac{1}{\sigma^2}(j-i)(\mu-j)+1\right)p_j \\
&=\sum_{j=0}^{\ell} \sum_{i=j}^{\ell} \left(\frac{1}{\sigma^2}(j-i)(\mu-j)+1\right)p_j \\
&=\sum_{j=0}^{\ell}\left\{\frac{1}{\sigma^2}(\mu-j) \sum_{i=j}^{\ell} \left(j-i\right)+\ell-j+1\right\}p_j \\
&=\sum_{j=0}^{\ell}\left\{\frac{1}{2\sigma^2}(j-\mu)(\ell-j)(\ell-j+1)\right\}p_j+\ell-\mu+1 \\
&=\frac{1}{2\sigma^2}\left\{-\sigma^2(2\ell+1)-\mu\E[Y^2]+\E[Y^3]\right\}+\ell-\mu+1\\
&=\frac{1}{2\sigma^2}\left\{-\sigma^2(2\ell+1+\mu)-\mu^3+\E[Y^3]\right\}+\ell-\mu+1\\
&=\frac{1}{2}-\frac{3\mu}{2}-\frac{\mu^3}{2\sigma^2}+\frac{1}{2\sigma^2}\E[Y^3]\\
&=\frac{1}{2}\left\{1+\frac{1}{\sigma^2}\E\left[(Y-\mu)^3\right]\right\}.\qedhere
\end{align*}
\end{proof}

\begin{proof}[Proof of  Lemma \ref{prop borne}]

For the bound on $|f_h'|$, see Proposition 2.14 in \cite{dobler2015stein}. 
Now, we consider the bound on $|(f'w)'|$. 
For $h\in \mathrm{Lip}(1)$, the equality $\T_{q,w} f_h=h-\E[h(Z)]$ can be rewritten as $f_h'w=-sf_h+h-\E[h(Z)]$, from which we deduce that $(f_h'w)'=-sf_h'-s'f_h+h'$. 
Take $x\in (a,b)$. By Lemma 5.3 (b) in \cite{dobler2015stein}, we know that
\begin{align*}
f_h(x)&=\frac{1}{(qw)(x)}\int_a^x (\E[h(Z)]-h)q=-\frac{1}{qw(x)}\left\{\bar{F}(x)\int_a^xh'F
+F(x)\int_x^b h' \bar{F}\right\}
\end{align*}
and by equation (47) in \cite{dobler2015stein}, we have
\begin{equation*}
f_h'(x)=\frac{1}{qw^2(x)}\left\{\left(qw+s\bar{F}\right)(x)\int_a^xh'F
-\left(qw-sF\right)(x)\int_x^b h' \bar{F}\right\}.
\end{equation*}
Putting the three last equalities together, we obtain
\begin{align*}
&\left|(f_h'w)'(x)\right|=\left|\left(-sf_h'-s'f_h+h'\right)(x)\right|\\
&=\left|\frac{-sqw-s^2\bar{F}+s'w\bar{F}}{qw^2}(x)\int_a^xh'F
+\frac{sqw-s^2F+s'wF}{qw^2}(x)\int_x^b h' \bar{F}+h'(x)\right|\\
&\leq\left|\frac{-sqw-s^2\bar{F}+s'w\bar{F}}{qw^2}(x)\right|\int_a^x|h'|F
+\left|\frac{sqw-s^2F+s'wF}{qw^2}(x)\right|\int_x^b |h'| \bar{F}+|h'(x)|.
\end{align*}
Since $h$ is Lipschitz, we get the desired result.
\end{proof}

\begin{proof}[Proof of Lemma \ref{prop borne 2}]
In this proof, we set $w=\tau_q$ and $s=\overline{\mathrm{Id}}$. Concerning equation \eqref{eq Gamma12}, see Lemma 5.2 in \cite{dobler2015stein}.
By Lemma \ref{prop borne}, we know that
\begin{align*}
|f'_h(x)|\leq \frac{1}{qw^2(x)}\left\{\left(qw+s\bar{F}\right)(x)\int_a^x F
+\left(qw-sF\right)(x)\int_x^b \bar{F}\right\}
\end{align*}
for $x\in (a,b)$. Observe that
\begin{align*}
\int_a^x F=xF(x)-\int_a^x t q(t)dt=-sF(x)+qw(x)
\end{align*}
and
\begin{equation*}
\int_x^b \bar{F}=\int_x^b t q(t)dt-x\bar{F}(x)=qw(x)+s\bar{F}(x).
\end{equation*}
So, we get the desired expression for $|f'_h|$. Now, we look at the bound on $|(f'_hw)'|$. We have
\begin{align*}
\left(\frac{s}{w}-\frac{s^2F}{qw^2}-\frac{F}{qw}\right)\Gamma_2&= \left(\frac{s}{w}+\frac{s^2\bar{F}}{qw^2}+\frac{\bar{F}}{qw}-\frac{s^2}{qw^2}-\frac{1}{qw}\right)(\Gamma_1+s)\\
&=\left(\frac{s}{w}+\frac{s^2\bar{F}}{qw^2}+\frac{\bar{F}}{wq}\right)\Gamma_1-\left(\frac{s^2}{qw^2}+\frac{1}{qw}\right)\Gamma_1\\
&\quad +\left(\frac{s}{w}-\frac{s^2F}{qw^2}-\frac{F}{qw}\right)s.
\end{align*}
Then, we compute
\begin{align*}
&-\left(\frac{s^2}{qw^2}+\frac{1}{qw}\right)\Gamma_1 +\left(\frac{s}{w}-\frac{s^2F}{qw^2}-\frac{F}{qw}\right)s\\
&\quad=-\left(\frac{s^2}{qw^2}+\frac{1}{qw}\right)\left(qw-sF \right)+\left(\frac{s}{w}-\frac{s^2F}{qw^2}-\frac{F}{qw}\right)s\\
&\quad=-\left(\frac{s^2}{qw^2}+\frac{1}{qw}\right)qw+\frac{s^2}{w}=-1.
\end{align*}
So, we have
\begin{equation*}
\left(\frac{s}{w}-\frac{s^2F}{qw^2}-\frac{F}{qw}\right)\Gamma_2=\left(\frac{s}{w}+\frac{s^2\bar{F}}{qw^2}+\frac{\bar{F}}{wq}\right)\Gamma_1-1.
\end{equation*}
and we obtain the new bound on $|(f'_hw)'|$. As a consequence, $|(f'_h w)'|\leq 2$ if 
$\frac{s}{w}+\frac{s^2\bar{F}}{qw^2}+\frac{\bar{F}}{wq}\geq 0$ (since $\Gamma_1>0$) and 
$\frac{s}{w}-\frac{s^2F}{qw^2}-\frac{F}{qw}\leq 0$ (since $\Gamma_2>0$). The first condition is equivalent to $F\leq \frac{sqw}{s^2+w}+1$ and the second one to $\frac{sqw}{s^2+w}\leq F$. The first equation is true if $\lim_{x\to b}sqw/(s^2+w)(x)= 0$ and (by taking the derivative on both sides)
\begin{equation}
\label{chose3}
q\geq \frac{-qw+s^2q}{s^2+w}+sqw\frac{2s-w'}{(s^2+w)^2}
=q\frac{-w^2+s^4+sw(2s-w')}{(s^2+w)^2}.  
\end{equation}
The condition on the limit in $b$ is always true because $s/(s^2+w)$ is bounded and $\lim_{x\to a}qw=0$, while the inequality \eqref{chose3} is equivalent to $-2w\leq sw'$. By a similar reasoning, we can show that $\frac{sqw}{s^2+w}\leq F$ is true if $-2w\leq sw'$. 
\end{proof}
 
\begin{proof}[Proof of Proposition \ref{prop borne3}]
We set $\mu=\E[Z]$. By Proposition \ref{prop borne 2}, we have to show that $-2\tau_q\leq \overline{\mathrm{Id}}\tau_q'$. 
This inequality translates into $-2\gamma-\mu\beta \leq (\beta+2\mu\alpha)x$ for all $x\in (a,b)$ when $Z\sim \mathrm{IP}(\alpha,\beta,\gamma)$. Assume first that the support of $q$ is $[a,b]$ with $a,b\in\R$. We know that $\tau_q(x)\geq 0$ for $x\in ]a,b[$ and $\tau_q(a)=0=\tau_q(b)$ by Proposition 2.1 (v) in \cite{Afendras}. Hence, we must have  $\tau_q(x)=\alpha(x-a)(x-b)=\alpha x^2-\alpha (a+b)x+\alpha ab$ with $\alpha<0$. So, in this case, we have $\beta=-\alpha(a+b)$ and $\gamma=\alpha ab$ and we have to check that
$$
\mu(a+b)-2ab\geq (2\mu-a-b)x
$$
for all $x\in [a,b]$. For $x=a$ this is equivalent to $(b-a)(\mu-a)\geq 0$ and for $x=b$ to $(b-\mu)(b-a)\geq 0$, which are both true since $a<\mu<b$. Hence, the inequality is verified for all $x\in [a,b]$.

Now assume that the support of $q$ is $[a,\infty[$ with $a\in\R$. We know that $\tau_q(x)> 0$ for $x>a$ and $\tau_q(a)=0$ by Proposition 2.1 (v) in \cite{Afendras}. Hence, we must have $\alpha>0$ and there exists  $c\in ]-\infty,a]$ such that $\tau_q(x)=\alpha(x-a)(x-c)=\alpha x^2-\alpha(a+c)x+\alpha ac$ with $\alpha>0$. So, in this case, we have $\beta=-\alpha(a+c)$ and $\gamma=\alpha ac$ and we have to check that
$$
\mu(a+c)-2ac\leq (2\mu-a-c)x
$$
for all $x\geq a$. Since $\mu\geq a\geq c$, we can replace $x$ by $a$ in the last inequality. This gives $0\leq (a-c)(\mu-a)$ which is true for the same reason. 
If the support of $q$ is $]-\infty,b]$ with $b\in\R$, we get the same conclusion by a similar reasoning. 


If the support of $q$ is $\R$, we have $-2\gamma-\mu\beta \leq (\beta+2\mu\alpha)x$ for all $x\in \R$, and thus $-2\tau_q\leq \overline{\mathrm{Id}}\tau_q'$, if and only if $\beta=-2\mu\alpha$ and $-\mu\beta/2=\alpha\mu^2\leq \gamma$. 
By Table 2.1 in \cite{Afendras}, we know that $\mu=0=\beta$, $\alpha=1/(n-1)=\gamma/n$ with $n>1$ when $X\sim t_n$ and $\beta=0=\alpha$ and $\gamma=\sigma^2$ when $X\sim \mathcal{N}(\mu,\sigma^2)$. For these values of the parameters, the conditions $\beta=-2\mu\alpha$ and $\alpha\mu^2\leq \gamma$ are true.
\end{proof}

\begin{proof}[Proof of Proposition \ref{prop borne sup}]
First, we show that (1) implies $\pi_i\leq 1$ for all $i=0,\ldots, \ell$. Observe that $\pi_i\leq 1$ if and only if $b_i:=\sum_{j=0}^{i-1}\left[(x_j-x_i)s_j+w_{j}\right]p_j\leq 0$ for $i=0,\ldots,\ell$ (with the convention $b_0=0$). Set $c_i:=p_{i}w_{i}- \delta_{i+1}\sum_{j=0}^is_jp_j$ and remark that 
\begin{align*}
b_i&=\sum_{j=0}^{i-1}\left(w_j-s_j\sum_{t=j+1}^i\delta_t\right)p_j=\sum_{j=0}^{i-1}w_j p_j-\sum_{t=1}^i\delta_t\sum_{j=0}^{t-1}s_jp_j\\
&=\sum_{k=0}^{i-1}\left(w_kp_k-\delta_{k+1}\sum_{j=0}^{k}s_jp_j\right)=\sum_{k=0}^{i-1}c_k.
\end{align*}
Since $b_0=0$ and $b_{\ell}=\pi_{\ell}w_{\ell}p_{\ell}-w_{\ell}p_{\ell}=-w_{\ell}p_{\ell}\leq 0$, we have $b_i\leq 0$ for all $i=1,\ldots,\ell$ if $b$ is decreasing or decreasing then increasing, $\text{i.e.}$ if 
there exist $i_1\in \left\{0,\ldots,\ell-1\right\}$ such that $c_i$ is negative for $i=0,\ldots,i_1$ and positive for $i=i_1+1,\ldots,\ell-1$.

Now we prove that (2) and (3) both entail (1). We set $d_0=p_0w_0/\delta_1-s_0p_0$ and $d_i:=p_{i}w_{i}/\delta_{i+1}-p_{i-1}w_{i-1}/\delta_i- s_{i}p_{i}$ for $i=1,\ldots,\ell-1$. We thus have $c_i/\delta_{i+1}=\sum_{j=0}^i d_j$. 
Assume that $c_{\ell-1}=w_{\ell-1}p_{\ell-1}+\delta_{\ell}s_{\ell}p_{\ell}\leq 0$. As $c_0=(w_0- \delta_1 s_0)p_0\leq 0$ and $c_{\ell-1}\leq 0$, we have $c_i\leq 0$ for all $i=0,\ldots,\ell-1$ if $c$ is monotone or decreasing then increasing, $\text{i.e.}$ if there exist $i_2\in\left\{0,\ldots,\ell-1\right\}$ such that $d_i$ is negative for $i=1,\ldots,i_2$ and positive for $i=i_2+1,\ldots,\ell-1$. 
We have thus $(2)\implies (1)$ with $i_1=\ell-1$. 

Now, assume that $c_{\ell-1}\geq 0$. Since $c_0\leq 0$ and $c_{\ell-1}\geq 0$, $c$ is negative then positive on $\left\{0,\ldots,\ell-1\right\}$ if $c$ is decreasing then increasing then decreasing, $\text{i.e.}$ if there exist $i_2,i_3\in\left\{0,\ldots,\ell-1\right\}$ such that $i_2\leq i_3$, $d_i$ is negative for $i=1,\ldots,i_2$ and $i=i_3+1,\ldots,\ell-1$ and positive for $i=i_2+1,\ldots,i_3$.
We thus have $(3)\implies (1)$ with $i_1\in [i_2,i_3]$.
\end{proof}

\begin{proof}[Proof of Proposition \ref{prop borne inf}]
First, we show that (1) implies $\pi_i\geq 0$ for all $i=0,\ldots, \ell$. Observe that $\pi_i\geq 0$ if and only if $b_i:=\sum_{j=0}^{i}\left[(x_j-x_i)s_j+w_{j}\right]p_j\geq 0$ for $i=0,\ldots,\ell$. Set $c_i:=p_{i}w_{i}- \delta_{i}\sum_{j=0}^{i-1}s_jp_j$ for $i=1,\ldots,\ell$ and remark that 
\begin{align*}
b_i=\sum_{j=0}^i\left(w_j-s_j\sum_{t=j+1}^i\delta_t\right)p_j=w_0p_0+\sum_{t=1}^i\left(w_tp_t-\delta_t\sum_{j=0}^{t-1}s_jp_j\right)=\sum_{t=0}^{i}c_t
\end{align*}
with the convention that $c_0=p_0w_0$. Since $b_0=p_0 w_0$ and $b_{\ell}=0$, we have $b_i\geq 0$  for all $i=i_1,\ldots,\ell$ if $b$ is decreasing or increasing then decreasing, i.e. if  there exist $i_1\in \left\{0,\ldots,\ell-1\right\}$ such that $c_i$ is positive for $i=1,\ldots,i_1$ and negative for $i=i_1+1,\ldots,\ell$. 

Now we prove that (2) and (3) both entail (1). We set $d_0=p_1w_1/\delta_1-s_0p_0$ and $d_i:=p_{i+1}w_{i+1}/\delta_{i+1}-p_{i}w_{i}/\delta_{i}- s_{i}p_{i}$ for $i=1,\ldots,\ell-1$. We thus have $c_i/\delta_{i}=\sum_{j=0}^{i-1} d_j$. 
Assume that $c_1=w_1p_1- \delta_1 s_0p_0\leq 0$. As $c_1\leq 0$ and $c_{\ell}=(w_{\ell}+\delta_{\ell}s_{\ell})p_{\ell}\leq 0$, we have $c_i\leq 0$ for all $i=1,\ldots,\ell$ if $c$ is monotone or decreasing then increasing, $\text{i.e.}$ if there exist $i_2\in\left\{0,\ldots,\ell-1\right\}$ such that $d_i$ is negative for for $i=1,\ldots,i_2$ and positive for $i=i_2+1,\ldots,\ell-1$. 
We have thus $(2)\implies (1)$ with $i_1=\ell-1$. 

Now, assume that $c_1\geq 0$. Since $c_1\geq 0$ and $c_{\ell}\leq 0$, $c$ is positive then negative on $\left\{1,\ldots,\ell\right\}$ if $c$ is increasing then decreasing then increasing, $\text{i.e.}$ if there exist $i_2,i_3\in\left\{0,\ldots,\ell-1\right\}$ such that $i_2\leq i_3$ and $d_i$ is positive for $i\in\left\{1,\ldots,i_2\right\}\cup\left\{i_3+1,\ldots,\ell-1\right\}$ and negative for $i\in\left\{i_2+1,\ldots,i_3\right\}$. We thus have $(3)\implies (1)$ with $i_1\in [i_2,i_3]$.
\end{proof}

\begin{proof}[Proof of Corollary \ref{coro panjer}]
 We have to check that the assumptions of Proposition \ref{prop panjer} hold. Set $\tau_i:=\tau_q(x_i)$ and $\overline{\mathrm{Id}}_i:=\overline{\mathrm{Id}}(x_i)$. By Proposition 2.1 in \cite{Afendras}, we know that $\tau_0=\tau_q(a)=0$ since $a\in\R$. The condition $\delta \overline{\mathrm{Id}}_0\geq \tau_0$ is thus directly verified. If $b<\infty$, we have also $\tau_q(b)=0$ by Proposition 2.1 in \cite{Afendras}. Take $\ell\in\N$ and observe that $\tau_q(a+\delta \ell)+\delta (\E[Z]-a-\delta \ell)$ is positive for $\delta=(\E[Z]-a)/\ell$ and negative for $\delta=(b-a)/\ell$. So, there exists $\delta>0$ such that $\tau_q(a+\delta \ell)+\delta (\E[Z]-a-\delta \ell)=0$. For $x>\E[Z]$, we know that $\tau_q$ is decreasing and $-\overline{\mathrm{Id}}$ increasing. Hence, $-\tau_q/\overline{\mathrm{Id}}$ is decreasing on $(\E[Z],b)$. 
 The condition $\tau_i+\delta \overline{\mathrm{Id}}_i\geq 0$ is thus satisfied for all $i\in I$. If $b=\infty$, the condition $\tau_i+\delta \overline{\mathrm{Id}}_i\geq 0$ is verified by definition of $\delta$ and we have to check that $\tau_q\in L^1(p)$. By the Raabe test, the sequence $(\tau_ip_i)$ is summable if $\lim_{i\to\infty}i\left(p_i\tau_i/(p_{i+1}\tau_{i+1})-1\right)>1$. We can compute that
$$
i\left(\frac{p_i\tau_i}{p_{i+1}\tau_{i+1}}-1\right)=i\left(\frac{\tau_i}{\tau_{i}+\delta \overline{\mathrm{Id}}_i}-1\right)=-\frac{i\delta \overline{\mathrm{Id}}_i}{\tau_{i}+\delta \overline{\mathrm{Id}}_i}=\frac{\delta i(\delta i+a-\E[Z])}{\tau_q(a+\delta i)+\delta \overline{\mathrm{Id}}(a+\delta i)}.
$$
The numerator and the denominator of the last expression are polynomials of degree two with leading term $\delta^2 i^2$ and $\alpha \delta^2 i^2$ respectively. Hence, the limit of this expression when $i\to\infty$ is $1/\alpha$. We can conclude with Proposition \ref{prop panjer}.  
\end{proof}

\begin{lemma}\label{boundlam}
Let $Z \sim \mathrm{Exp}(\lambda)$ be an exponential variable with parameter $\lambda$. We have $||f'_h||_{\infty}\leq \lambda$ for all $h\in \mathrm{Lip}(1)$. 
\end{lemma}
\begin{proof}
Set $w(x)=\tau_q(x)=x/\lambda$ and $G(x)=2\left(qw+s\bar{F}\right)\left(qw-sF\right)/qw^2(x)=2(e^{-\lambda x} + \lambda x -1)/(\lambda x^2)$. By Proposition \ref{prop borne 2}, we know that $||f'_h||_{\infty}\leq ||G||_{\infty}$ for all $h\in \mathrm{Lip}(1)$. 
Applying two times the Hospital' theorem, we get
\begin{align*}
\lim_{x\to 0}G(x)=\lim_{x\to 0}2\frac{1 - e^{-\lambda x} - \lambda x}{\lambda x^2}=\lim_{x\to 0}\frac{e^{-\lambda x}-1}{x}=\lim_{x\to 0}\lambda e^{-\lambda x}=\lambda.
\end{align*}
Hence, if we show that $G$ is decreasing, we will obtain $||G||_{\infty}= \lim_{x\to 0}G(x)=\lambda$. Since $G'(x)=(2 -\lambda x- e^{-\lambda x}(2+\lambda x))/(\lambda x^3)$,  $G'\leq 0$ is equivalent to $G_2(x)=2 -\lambda x- e^{-\lambda x}(2+\lambda x)\leq 0$. As $G_2(0)=0$, it is enough to show that $G_2$ is decreasing. We have $G_2'(x)=-\lambda+e^{-\lambda x}(\lambda+\lambda^2x)$ and thus $G_2'(0)=0$. Once again, it's enough to show that $G_2'$ is decreasing, which is true since $G_2''(x)=-\lambda^3xe^{-\lambda x}\leq 0$.

\end{proof}

\begin{definition}
The constants $b_0$ and $b_1$ from \cite[Lemma 3.4]{goldstein2013stein} are 
\begin{equation}
    \label{eqn:b0}
    b_0(\alpha, \beta) = \begin{cases}
    & 4 \max(|\alpha-1|,|\beta-1|)\mbox{ if }\alpha\leq 2,\ \beta\leq 2\\
    & (\alpha+\beta-2)\max\left(\frac{\alpha-1}{(\alpha-2)^2},\frac{|\beta-1|}{\beta^2}\right)\mbox{ if }\alpha\geq 2,\ \beta\leq 2\\
    & (\alpha+\beta-2)\max\left(\frac{|\alpha-1|}{\alpha^2},\frac{\beta-1}{(\beta-2)^2}\right)\mbox{ if }\alpha\leq 2,\ \beta\geq 2\\
    & (\alpha+\beta-2)\max\left(\frac{1}{\alpha-1},\frac{1}{\beta-1}\right)\mbox{ if }\alpha\geq 2,\ \beta\geq 2
    \end{cases}
\end{equation}
and 
\begin{equation}
    \label{eqn:b1}
    b_1(\alpha, \beta) = \begin{cases}
   & 4 \left(1+\frac{\max(\alpha,\beta)}{\alpha+\beta}\right)\mbox{ if }\alpha\leq 2,\ \beta\leq 2\\
    & \frac{(\alpha+\beta-2)^2}{\min(\alpha-2,\beta)^2}+2\max\left(\frac{\alpha}{\alpha-2},1\right)\mbox{ if }\alpha\geq 2,\ \beta\leq 2\\
    & \frac{(\alpha+\beta-2)^2}{\min(\alpha,\beta-2)^2}+2\max\left(1,\frac{\beta}{\beta-2}\right)\mbox{ if }\alpha\leq 2,\ \beta\geq 2\\
    & \frac{(\alpha+\beta-2)^2}{\min(\alpha-1,\beta-1)^2}+2\max\left(\frac{\alpha}{\alpha-1},\frac{\beta}{\beta-1}\right)\mbox{ if }\alpha\geq 2,\ \beta\geq 2
    \end{cases}
\end{equation}
\end{definition}

\begin{proof}[Proof of Lemma \ref{lem beta bound}]
In this proof, we write $w=\tau_q$ and $s=\overline{\mathrm{Id}}$. We have to prove inequality \eqref{eq lem beta}. We know that $|w'f_h'|\leq 2|w'|/(qw^2)\Gamma_1\Gamma_2$ on $]0,1[$ for all $h\in \text{Lip}(1)$ by Lemma \ref{prop borne 2}. We consider the subintervals $]0,1/2]$ and $]1/2,1[$ separately. We have $\Gamma_2\leq s(0)=\alpha/(\alpha+\beta)$ and we want to show that $w'/(qw^2)\Gamma_1\leq w'(0)/(s(0)^2+w'(0)s(0))=(\alpha+\beta)/(\alpha+\alpha^2)=:M_0$ on $]0,1/2]$. This is equivalent to $\Gamma_1-M_0qw^2/w'\leq 0$. As it is true in 0, it is also true on $]0,1/2]$ if the expression is decreasing, i.e. if 
$$
F-M_0\left(\frac{qsw}{w'}+qw-\frac{qw^2w''}{(w')^2}\right)\leq 0
$$ 
on $]0,1/2]$. In turn, as this is true in 0, it is true on $]0,1/2]$ if
$$
q-M_0q\left(\frac{s^2}{w'}-\frac{w}{w'}-2\frac{sww''}{(w')^2}+s-\frac{ww''}{w'}+2\frac{w^2(w'')^2}{(w')^3}\right)\leq 0
$$
or, equivalently,
$$
(w')^3-M_0\left((sw')^2-w(w')^2-2sww'w''+s(w')^3-w(w')^2w''+2w^2(w'')^2\right)\leq 0
$$
on $]0,1/2]$. In the case of the Beta distribution, this inequality becomes
\begin{align*}
g_0(x):=&2 (1 + \alpha) (-1 + \beta) - (-2 + \alpha^2 + \beta (7 + \beta) + 5 \alpha (-1 + 2 \beta)) x\\ 
&+ 4 (\alpha^2 + 4 \alpha \beta + \beta (2 + \beta)) x^2 - 
   4 (\alpha + \beta) (1 + \alpha + \beta) x^3\leq 0
\end{align*}
on $]0,1/2]$. We have $g_0(0)=2 (1 + \alpha) (-1 + \beta)$ and $g_0(1/2)=-1$, which are both negative since $\beta\leq 1$, and $g_0'(1/2)=2 (1 + \alpha - \beta)\geq 0$. The second derivative $g_0''(x)=8 (\alpha^2 + 4 \alpha \beta + \beta (2 + \beta))  - 
   24 (\alpha + \beta) (1 + \alpha + \beta) x$ is either positive or positive then negative on $]0,1/2]$. In both cases, these observations are sufficient to ensure that $g_0\leq 0$ on $]0,1/2]$. Combining the two bounds, we obtain that $w'/(qw^2)\Gamma_1\Gamma_2\leq 1/(1+\alpha)$ on $]0,1/2]$.
   
Now, we want to show that $-w'/(qw^2)\Gamma_2\leq (\alpha+\beta)/(\beta+\beta^2)=:M_1$ on $]1/2,1[$. We proceed in a similar way. It is equivalent to $\Gamma_2+M_1 qw^2/w'\leq 0$. As it is true in 1, it is also true on $]1/2,1]$ if the expression is increasing, i.e.  if
$$
-\bar{F}+M_1\left(\frac{qsw}{w'}+qw-\frac{qw^2w''}{(w')^2}\right)\geq 0
$$ 
on $]1/2,1]$. In turn, as this is true in 1, it is true on $]1/2,1]$ if
$$
q+M_1q\left(\frac{s^2}{w'}-\frac{w}{w'}-2\frac{sww''}{(w')^2}+s-\frac{ww''}{w'}+2\frac{w^2(w'')^2}{(w')^3}\right)\geq 0
$$
or, equivalently,
$$
(w')^3+M_1\left((sw')^2-w(w')^2-2sww'w''+s(w')^3-w(w')^2w''+2w^2(w'')^2\right)\geq 0
$$
on $]1/2,1]$. In the case of the Beta distribution, this inequality becomes
\begin{align*}
g_1(x):=&2 (1 - \alpha) (1 + \beta)
+ (2 - \alpha (7 + \alpha) + 5 \beta - 10 \alpha \beta - \beta^2) (x-1)\\
&-4 (\alpha (2 + \alpha) + 4 \alpha \beta + \beta^2) (x-1)^2 
 -4 (\alpha + \beta) (1 + \alpha + \beta) (x-1)^3\geq 0
\end{align*}
on $]1/2,1]$. We have $g_1(0)=2 (1 - \alpha) (1 + \beta)$ and $g_1(1/2)=1$, which are both positive since $\alpha\leq 1$, and $g_1'(1/2)=2 (1 -\alpha+\beta)\geq 0$. The second derivative $g_1''(x)=-8 (2 \alpha + \alpha^2 + 4 \alpha \beta + \beta^2) -24 (\alpha + \beta) (1 + \alpha + \beta) (x-1)$ is either negative or positive then negative on $]1/2,1[$. In both cases, these observations are sufficient to ensure that $g_1\geq 0$ on $]1/2,1[$. Using the inequality $\Gamma_1\leq -s(1)=\beta/(\alpha+\beta)$, we obtain that $-w'/(qw^2)\Gamma_1\Gamma_2\leq 1/(1+\beta)$ on $]1/2,1[$.
\end{proof}

\begin{proof}[Direct proof of inequality \eqref{eq:claim01}]
On the one hand, since 
\begin{align*}
&An\left(B (B + n - \sqrt{n (A + B + n)}) + 
   A (B - n + \sqrt{n (A + B + n)})\right)\\
   &\quad=\left(An + 
   B \sqrt{n (A + B + n)}\right) A\left( 
    - n + \sqrt{n (A + B + n)}\right)\geq 0,
\end{align*}
we have $\pi_i\geq 0$ for all $i=0,\ldots,n$. On the other hand, the inequality $\pi_i\leq 1$ is equivalent to 
\begin{equation*}
i (B - i + n) \left(-A (A + B) + (A-B)\left( \sqrt{n (A + B + n)}-n\right)\right)\leq 0.
\end{equation*}
This polynomial in $i$ is negative between its roots $0$ and $n+B$, and so between $0$ and $n$, if and only if $A (A + B) + (B-A)\left( \sqrt{n (A + B + n)}-n\right)\geq 0$. 
If $B\geq A$, it is trivially true. Otherwise, we have
\begin{align}
\label{eq inequality}
\frac{\sqrt{n (A + B + n)}-n}{A+B}&=\frac{1}{A+B} \int_0^{A+B}\left(\sqrt{n(x+n)}\right)'dx \nonumber\\
\nonumber
&=\frac{1}{A+B}\int_0^{A+B}\frac{1}{2}\frac{n}{\sqrt{n(x+n)}}dx\\
&\leq \frac{1}{2}\leq \frac{A}{A-B}.
\end{align}
\end{proof}

\begin{proof}[Alternative proof of inequality \eqref{eq:claim01}]
In this proof, we write $w=\tau_q$ and $s=\overline{\mathrm{Id}}$. We begin by showing that $\pi_i\geq 0$ for all $i=0,\ldots, n$ using Propositions \ref{prop borne inf}. The necessary condition $w_{\ell}\leq -\delta s_{\ell}$ becomes 
$$
B(A+B)+(A-B)\left(\sqrt{n (A + B + n)}-n\right)\geq 0,
$$
which is true as we have shown in \eqref{eq inequality}, and the condition $w_1\geq\delta s_0 p_0/p_1$ takes the form
$$
(An-B-A)\left(B(A+B)+(A-B)\left(\sqrt{n (A + B + n)}-n\right)\right)\geq 0,
$$
which is true when $n\geq (A+B)/A$. This means that we can use the third point of Proposition \ref{prop borne inf}. 
We can compute that $p_{i+1}w_{i+1}/p_i-w_i-\delta s_i$ is equal to
\begin{align*}
&\left(B (A+B) + (A - B) \left(\sqrt{n (A + B + n)}-n\right)\right)\\
   &\left((-1 + B) B i (1 + i) + 
   A^2 (i - n) (1 + i - n) + A (i - n) ((-1 + 2 B) (1 + i) + n)\right)\\
   &\left((A + B)^3 (1 + i) n (A + B + n) (B + n-i-1)\right)^{-1}
\end{align*}
with $i=0,\ldots,n-1$. 
This expression has the same sign as the polynomial
$$
g_0(i):=(-1 + B) B i (1 + i) + 
   A^2 (i - n) (1 + i - n) + A (i - n) ((-1 + 2 B) (1 + i) + n)
$$
for $i\in [0,n-1]$. 
The extremum of $g_0$ is 
$$
-\frac{1}{4} (-1 + A + B) (A + B) - A B n - \frac{A B n^2}{A + B}
$$
and is reached in $i= A n/(A + B)-1/2$. Remark that the extremum of $g_0$ is negative as soon as $n\geq 1/(4\min(A,B))$ and it is reached in $[0,n-1]$ as soon as $n\geq (A+B)/(2\min(A,B))$. Hence, when $n$ is large enough, there are three possibilities. Either $g_0$ doesn't change sign in $[0,n-1]$, either $g_0$ changes sign only once, or $g_0$ is positive then negative then positive in $[0,n-1]$. In the three cases, the conditions of Proposition \ref{prop borne inf} are verified. 

We can show that $\pi_i\leq 1$ for all $i=0,\ldots,n$ with Proposition \ref{prop borne sup} in a similar way. The necessary condition $w_{0}\leq \delta s_{0}$ becomes 
$$
A(A+B)+(B-A)\left(\sqrt{n (A + B + n)}-n\right)\geq 0,
$$
which is true as we have already shown, and the condition $w_{\ell-1}\geq-\delta s_{\ell} p_{\ell}/p_{\ell-1}$ takes the form
$$
(Bn-B-A)\left(A(A+B)+(B-A)\left(\sqrt{n (A + B + n)}-n\right)\right)\geq 0,
$$
which is true when $n\geq (A+B)/B$. This means that we can use the third point of Proposition \ref{prop borne sup}. 
We can compute that $w_i-w_{i-1}p_{i-1}/p_i-\delta s_i$ 
has the same sign as the polynomial
$$
g_1(i):=(1 - A - B) (A + B) i^2 - 
 i (1 - A - B) (A + B+2A n) + ( 1-A) A (1+n)n 
$$
for $i\in [1,n]$. 
The extremum of $g_1$ is 
$$
\frac{1}{4} (-1 + A + B) (A + B) + A B n + \frac{A B n^2}{A + B}
$$
and is reached in $i= A n/(A + B)+1/2$. Remark that the extremum of $g_1$ is positive as soon as $n\geq 1/(4\min(A,B))$ and it is reached in $[1,n]$ as soon as $n\geq (A+B)/(2\min(A,B))$. Hence, when $n$ is large enough, there are three possibilities. Either $g_1$ doesn't change sign in $[1,n]$, either $g_1$ changes sign only once, or $g_1$ is negative then positive then negative in $[1,n]$. In the three cases, the conditions of Proposition \ref{prop borne sup} are verified. 
\end{proof}

\begin{proof}[End of proof of Proposition \ref{semici}]
    All that remains is to prove the weights $\pi_i$ lie in $[0,1]$.  On the one hand, since $4 i^2 - 4 in + n + 2$ is a polynomial of order 2 in $i$, it is negative between its roots $(n-\sqrt{n^2-n-2})/2$ and $(n+\sqrt{n^2-n-2})/2$. It is easy to see that $(n-\sqrt{n^2-n-2})/2\leq 1$ and $n-1\leq (n+\sqrt{n^2-n-2})/2$ for all $n\geq 2$. So,
we have $\pi_i\geq 0$ for all $i=1,\ldots,n-1$. On the other hand, the inequality $\pi_i\leq 1$ is equivalent to 
\begin{equation*}
(-1 + i) (-1 + 2 i - 2 n)\leq 0,
\end{equation*}
which is true for all $i=1,\ldots,n-1$. 
\end{proof}

\begin{proof}[Proof of Proposition \ref{prop:moran}]
 By Corollary 2.3 and Proposition 2.4 in \cite{fulman2023beta} we know that 
   $$\E[Y]= \frac{2na}{a+b} \mbox{  and }\Var[Y]=\frac{8n^3 a b}{(a+b)^2(2n(a+b+1)-a-b)}$$
  (our $Y$ is $2n W$ in the notations from \cite{fulman2023beta}). The standardization of $X$ follows from \eqref{eq:autosta}. Unfortunately we have not been able to obtain an expression for the weights $\pi_i$; we resort to using the propositions from Section \ref{sn:more}.  We begin by showing that $\pi_i\geq 0$ for all $i=0,1,\ldots, 2n$. The necessary condition of Proposition \ref{prop borne inf} is $\Var[Y]\tau_q(x_{2n})\leq \Var[Z](2n-\E[Y])$, which is equivalent to
$$
\frac{(a - b) b}{(a + b)^3} + \frac{b}{4 (1 + a + b) n^2} - \frac{ b (a + a^2 + 2 a b + b^2)}{2 (a + b)^2 (1 + a + b) n}
+  \sqrt{1-\frac{a + b}{2 (1 + a + b) n} } \frac{(b - a) b}{(a + b)^3}\leq 0.
$$
As 
$$
\left(1- \sqrt{1-\frac{a + b}{2 (1 + a + b) n} }\right)n\to (a+b)/(4(1+a+b)),
$$
this expression behaves as $b (a + b - (1+2a + 2b) n)/(4 (a + b) (1 + a + b) n^2)$ when $n\to \infty$, which is indeed negative. The condition $w_1-\delta s_0 p_0/p_1\geq 0$ is equivalent to 
\begin{align*}
&\frac{(a + b)^4 - 2 (a + b)^3 (1 + 2 a + b) n +  4 (a + b)^2 (b + a (a + b)) n^2}{16 (1 + a + b) n^4} \\
&+  \frac{8 (a + b) (a + 2 a^2 - b (1 + b)) n^3 - 16 a (a - b) (1 + a + b) n^4}{16 (1 + a + b) n^4} \\
&+ \sqrt{1-\frac{a + b}{2 (1 + a + b) n} } \frac{(a - b) (-b + a (-1 + 2 n))}{2 n}\geq 0.
\end{align*}
When $n\to\infty$, this expression behaves as
$$
\frac{(a + b)^2 - 2 (a + b) (1 + 2 a + b) n + 2 (1 + 2 a) (a + b) n^2 + 
 4 a n^3}{16 (a + b) (1 + a + b) n^4}
$$
which is positive for $n$ large enough. Thus, we can use statement (3) of Proposition \ref{prop borne inf}. We can compute that, as soon as $2n\geq a+b$, $p_{i+1}w_{i+1}/p_i-w_i-\delta s_i$ has the same sign as a polynomial $g$ of order 3 with 
\begin{align*}
g'''=\frac{(a + b)^4 (2n-1 + 1/(1 + a + b))}{8 n^3}\geq 0.
\end{align*}
Hence, $g$ is either convex or concave or concave then convex.
The leading terms of $g(0)$ are
\begin{align*}
&\frac{a}{1 + a + b} \left(2 (-a + a^3 + 2 a^2 (-1 + b) + b - b^3) n - 4 (-1 + a) (a - b) (1 + a + b) n^2\right) \\
&+ \sqrt{1-\frac{a + b}{2 (1 + a + b) n} }a (a - b)  (-2 (-1 + b) n + 4 (-1 + a) n^2)
\end{align*}
This expression behaves as
$$
\frac{a (a + b) ((a - b) (-1 + b) + 2 (-1 + a) (a + b) n)}{2 (1 + a + b)}
$$
when $n\to\infty$. Hence, for $n$ large enough, $g(0)$ is negative when $a<1$ and positive when $a>1$. 

Hence, if $a>1$ and for $n$ large enough, $g(0)\geq 0$ and $g$ is either convex or concave or concave then convex. In all cases, $g$ can change sign at most twice. The conditions of Proposition \ref{prop borne inf} are thus verified and we can conlude that $p_i\geq 0$ for all $i=0,1,\ldots,2n$.

Now, we show that $\pi_i\leq 1$ for all $i=0,1,\ldots,2n$ with Proposition \ref{prop borne sup} in a similar way. The necessary condition of Proposition \ref{prop borne sup} is $\Var[Y]w(x_{0})\leq \Var[Z]\E[Y]$, which is equivalent to
\begin{align*}
&\frac{a ((a + b)^3 - 2 (a + b) (b + (a + b)^2) n - 
    4 (a - b) (1 + a + b) n^2)}{4 (a + b)^3 (1 + a + b) n^2} \\
&+ \sqrt{1-\frac{a + b}{2 (1 + a + b) n}} \frac{
   a (a - b) }{(a + b)^3}\leq 0.
\end{align*}
This expression behaves as $a (a + b - (1+2a +2b) n)/(4 (a + b) (1 + a + b) n^2)$ when $n\to \infty$, which is indeed negative. The condition $w_{2n-1}\geq -\delta s_{2n} p_{2n}/p_{2n-1}$ is equivalent to 
\begin{align*}
&\frac{-(a + b)^4 + 2 (a + b)^3 (3 + a + 2 b) n - 4 (a + b)^2 (2 + a (3 + b) + b (6 + b)) n^2}{16 (a + b)^3 (1 + a + b) n^4}\\
&+ \frac{  8 (a + b) (a + a^2 + 4 a b + b (3 + 4 b)) n^3 + 
  16 (a - b) b (1 + a + b) n^4}{16 (a + b)^3 (1 + a + b) n^4} \\
 & + \sqrt{1-\frac{a + b}{2 (1 + a + b) n}}\frac{(a - b) (a + b - 2 b n)}{2 (a + b)^3n}\geq 0.
\end{align*}
When $n\to\infty$, this expression behaves as
$$
-\frac{(a + b)^2 - 2 (a + b) (3 + a + 2 b) n + 
  2 (4 + a (7 + 2 b) + b (11 + 2 b)) n^2 - 4 (4 + 4 a + 5 b) n^3}{
 16 (a + b) (1 + a + b) n^4}
$$
which is positive for $n$ large enough.  Thus, we can use statement (3) of Proposition \ref{prop borne sup}. We can compute that, as soon as $2n\geq a+b$, $w_{i}-w_{i-1}p_{i-1}/p_i-\delta s_i$ has the same sign as a polynomial $g$ of order 3 with 
\begin{align*}
g'''=-\frac{(a + b)^4 (a + b - 2 (1 + a + b) n)}{8 (1 + a + b) n^3}\geq 0.
\end{align*}
Hence, $g$ is either convex or concave or concave then convex. 
We have $g(2n)$ equal to
\begin{align*}
   &-\frac{b ((a + b)^4 - 2 (a + b)^2 (1 + a + b) (a + 2 b) n + 
    4 (a + b) (-a + (2 + a + a^2) b + 2 (1 + a) b^2 + b^3) n^2 )}{4 (1 + a + b) n^2}\\
    &-\frac{ b(8 (a - a^3 + 2 a b^2 + b (-1 + (-2 + b) b)) n^3 - 
    16 (-1 + b) (-a + b) (1 + a + b) n^4)}{4 (1 + a + b) n^2}\\
    &- b (-a + b) (-a (1 + 2 n) + (-1 + 2 n) (b + 2 (-1 + b) n)) \sqrt{1-\frac{a + b}{2 (1 + a + b) n} } 
\end{align*}
When $n\to\infty$, this expression behaves as
$$
-\frac{b (a + b)^2 ((a + b)^2 - (a + b) (3 + 2 a + 4 b) n + 
   2 (1 + a + 2 a b + 2 b (1 + b)) n^2 + 4 (-1 + b) n^3)}{4 (1 + a + 
   b) n^2}
$$
Hence, for $n$ large enough, $g(2n)$ is negative when $b>1$ and positive when $b<1$. 

Therefore, if $b>1$ and for $n$ large enough, $g(2n)\leq 0$ and $g$ is either convex or concave or concave then convex. This means that $g$ can change sign at most twice, and is negative on the boundaries in such case. The conditions of Proposition \ref{prop borne sup} are thus verified and we can conclude that $\pi_i\leq 1$ for all $i=0,1,\ldots,2n$. The result follows by Theorem \ref{theo2dform}.
\end{proof}

\end{document}